\documentclass[11pt]{amsart}
\usepackage{graphics}
\usepackage[english]{babel}
\usepackage{amssymb,amsmath,amsfonts,mathrsfs}

\usepackage[utf8]{inputenc}
\usepackage[usenames]{xcolor}
\usepackage{ulem}

\newcommand{\mic}[1]{\textcolor{violet}{#1}}
\newcommand{\comment}[1]{}
   
    \newcommand{\set}[1]{{\left\{#1\right\}}}
\newcommand{\pa}[1]{{\left(#1\right)}}
\newcommand{\sq}[1]{{\left[#1\right]}}

\newcommand{\abs}[1]{{\left|#1\right|}}
\newcommand{\norm}[1]{{\left \|#1\right \|}}

\newcommand{\T}{\mathbb{T}}
\newcommand{\Z}{\mathbb{Z}}
\newcommand{\R}{\mathbb{R}}
\newcommand{\C}{\mathbb{C}}
\newcommand{\teta}{\theta}
\newcommand{\dg}{{\mathtt{D}_\g}}

\newcommand{\im}{I}
\newcommand{\na}{\widehat{n}}

\newcommand{\id}{\operatorname{id}}

\newcommand{\ad}{\operatorname{ad}}

\usepackage{amsthm}

   \newtheorem*{thm*}{Theorem}

\newtheorem{thm}{Theorem}[section]
\newtheorem{cor}[thm]{Corollary}
\newtheorem{prop}[thm]{Proposition}
\newtheorem{lemma}[thm]{Lemma}
\newtheorem{rmk}[thm]{Remark}

\newtheorem{ex}[thm]{Example}

\theoremstyle{definition}

\newtheorem{defn}[thm]{Definition}

\newcommand{\g}{\gamma}

\newcommand{\s}{{\sigma}}

\newcommand{\wtZ}{\widetilde{Z}}

\def\wc{ {}}

\newcommand{\om}{{\omega}}

\newcommand{\N}{{\mathbb N}}

\newcommand{\cC}{{\mathcal C}}

\newcommand{\cF}{{\mathcal F}}
\newcommand{\cG}{{\mathcal G}}
\newcommand{\cH}{{\mathcal H}}

\newcommand{\cK}{{\mathcal K}}

\newcommand{\cM}{{\mathcal M}}

\newcommand{\cR}{{\mathcal R}}

\newcommand{\td}{{\mathtt{d}}}

\newcommand{\tn}{{\mathtt{n}}}
\newcommand{\tr}{{\mathtt{r}}}

\newcommand{\tC}{{\mathtt{C}}}
\newcommand{\tD}{{\mathtt{D}}}

\newcommand{\tK}{{\mathtt{K}}}

\newcommand{\tN}{{\mathtt{N}}}

\newcommand{\be}{{\bf e}}

\newcommand{\bu}{{\bf u}}

\usepackage{bm}

\newcommand{\al}{{\alpha}}
\newcommand{\bt}{{\beta}}

\renewcommand{\im}{{\rm i}}
\newcommand{\jap}[1]{\langle #1 \rangle}
\newcommand{\und}[1]{\underline{#1}}

\newcommand{\e}{{\varepsilon}}

\renewcommand{\th}{{\mathtt{h}}}
\newcommand{\nnorm}[1]{{\left\vert\kern-0.25ex\left\vert\kern-0.25ex\left\vert #1 
    \right\vert\kern-0.25ex\right\vert\kern-0.25ex\right\vert}}

\newcommand{\bcoeffu}[1]{#1_{\bal^{1},\bbt^{1}}}
\newcommand{\bcoeffd}[1]{#1_{\bal^{2},\bbt^{2}}}

\newcommand{\buu}{{u^\bal}{\bar{u}^\bbt}}

\newcommand{\es}{e^{\set{S,\cdot}}}
\newcommand{\bal}{{ \al}}
\newcommand{\bbt}{{ \bt}}
\newcommand{\baluno}{\bal^{(1)}}
\newcommand{\baldue}{\bal^{(2)}}
\newcommand{\bbtuno}{\bbt^{(1)}}
\newcommand{\bbtdue}{\bbt^{(2)}}

\newcommand{\suca}{\mathtt N}
\newcommand{\ri}{r}
\newcommand{\rs}{{r^*}}

\newcommand{\Cuno}{{\mathcal C_1}}

\newcommand{\rp}[1]{Proposition~\ref{#1}}

\definecolor{indianyellow}{rgb}{0.89, 0.66, 0.34}
\definecolor{goldenpoppy}{rgb}{0.99, 0.76, 0.0}
\definecolor{gamboge}{rgb}{0.89, 0.61, 0.06}
\definecolor{frenchbeige}{rgb}{0.65, 0.48, 0.36}
\definecolor{burntorange}{rgb}{0.8, 0.33, 0.0}

\author{Michela Procesi$^{\dag}$}
\address{Department of Mathematics and Physics,
	Roma 3}
\email{procesi@mat.uniroma3.it}

\author{Laurent Stolovitch$^{\dag\dag}$}
\address{CNRS and Laboratoire J.-A. Dieudonn\'e
	U.M.R. 7351, Universit\'e C\^ote d'Azur, Parc Valrose
	06108 Nice Cedex 02, France}
\email{stolo@unice.fr}
\thanks{ $^{\dag\dag}$Research of L. Stolovitch was supported by the French government, through the UCAJEDI	Investments in the Future project managed by the National Research Agency (ANR) with the reference	number ANR-15-IDEX-01.}

\begin{document}

 \title{About linearization of infinite-dimensional Hamiltonian systems}

\begin{abstract}
This article is concerned with analytic Hamiltonian dynamical systems in infinite dimension in a neighborhood of an elliptic fixed point. Given a quadratic Hamiltonian, we consider the set of its analytic higher order perturbations. We first define the subset of elements which are formally symplectically conjugacted to a (formal) Birkhoff normal form. We  prove that if the quadratic Hamiltonian satisfies a Diophantine-like condition and if such a perturbation is formally symplectically conjugated to the quadratic Hamiltonian, then it is also analytically symplectically conjugated to it. Of course what is an analytic symplectic change of variables depends strongly on the choice of the phase space. Here we  work on periodic functions with Gevrey regularity.
\end{abstract} 
\maketitle
\setcounter{tocdepth}{2} 

\section{Introduction}\label{intro}
In finite dimension, studying the behaviour of the orbits of a vector field (or of diffeomorphism) nearby a fixed point is a fundamental and classical problem. The very first natural step into this understanding is to compare the dynamical system with its linearization at the fixed point. This is done by trying to transform the dynamical system into its linear part by a change of coordinates. There are formal obstructions to do so, called resonances. Hence, in general, one can merely expect the dynamical system to be transformed into a {\it normal form}, that is supposed to capture effect the very nonlinearities, through a formal change of coordinates. It was understood by the end of the 19th century that if the convex hull of the eigenvalues of the linear part does not contain the origin (one says then that the linear part is in the "Poincaré domain"), and if an higher order analytic perturbation is formally conjugate to the linear part, then it is also analytically so. When the linear part does not satisfy this property, then one has so-called "small divisors" that may forbid the transformation to be analytic. It was a major step forward made by C.L. Siegel \cite{siegel}, followed by H. R\"ussmann \cite{russmann-ihes}(for diffeomorphisms) and by A.D. Brjuno \cite{bruno} (for vector fields) who devised a sufficient "small divisors condition" ensuring the analycity of a linearizing transformation as soon as there exists a formal one. 
Linearizing (resp. Normalizing) problems for diffeomorphisms  were devised by J. P\"oschel \cite{Posch} and for commuting families by the second author \cite{stolo-diffeos} (resp. \cite{Stolo-ihes}). 
By the end of the 70's, it became clear to few people that some PDE's problems could be translated into an infinite dimensional dynamical systems to which one would have tried to apply methods of finite dimension. In particular, we mention the work by E. Zehnder \cite{zehnder-infinite} and V. Nikolenko \cite{Nik} who gave results similar to finite dimensional ones. It happens that the "small divisors condition" they required are too strong and are rarely satisfied. Furthermore, in general, the notion of formal normal form and formal change of variables should be clarified (for instance  if one defines formal polynomials and formal power series it is not in general true that this space has a Poisson algebra structure). Neverteless, in some very peculiar situation, this problem can be handled\cite{stolo-bambusi}. 
\\
Starting from the mid 80', there has been a  lot of interest {in studying long time behavior of solutions of PDEs. For those PDEs which can be considerered as Hamiltonian (infinite dimensional) dynamical systems related to a symplectic sturucture, one natural way to proceed is to prove the existence of finite dimensional invariant tori in the phase space. This usually implies the existence of quasi-periodic solutions, which are defined for all time. Lot of progresses has been done on the problem of extending KAM theory to PDEs. This circle of problems are very related, though distinct, to the ones solved in this article. Indeed, here one considers a dynamical system close to an elliptic  fixed point with the purpose of conjugating it to its most simple {\it normal form}~: its linear part at the fixed point. On the other hand, in KAM theory, one looks for the existence of a { \it finite dimensional  invariant  flat torus} on which the dynamics is the linear translation by a diophantine frequency. There is by now a wide literature dealing the subject related to semilinear PDEs, starting from \cite{Kuk,Po,KP, W,CW}, (for instance, see  \cite{EK,GYX, PPBumi, BKM18,Y} for more recent treatments). It has been early understood that these results might be seen through elaborated versions of  "Nash-Moser" theorem see for instance \cite{Bo98,BB3,BCP,CM}. 
	We finally mention \cite{FGPr, BBHM,BM,FG} for the case of fully-nonlinear PDEs. See also \cite{BMP:almost,CY} and references therein for infinite-dimensional tori. }

Birkhoff normal form (BNF) methods have been used in order to prove long time existence results and control of Sobolev norms for many classes of evolution PDEs close to an elliptic fixed point. Loosely speaking the point is    to  canonically transform $H$ into a Hamiltonian Normal form which depends only on the actions  plus a remainder term whose the Taylor polynomial, at the origin is  of degree $\tN+1 $. If one achieves this  then initial data which are $\delta$-small (with respect to the norm on the phase space) stay small (in the same norm) for times of order $\delta^{-\suca}$. A more precise formulation is given in the {\it Strategy section }below. Of course in the infinite dimensional setting this stability time depends strongly on the choice of the phase space as well as on the nature of the non-linear terms. A further problem is that in general it is not obvious that one can perform even one step of this procedure,  indeed the generating function of the desired change of variables is a {\it formal} polynomial which in infinite dimension is not necessarily analytic. This is a particularly difficult problem in the case of PDEs with derivatives in the nonlinearity.
\\
Let us briefly describe some of the literature.  Regarding applications to PDEs (and particularly the NLS) the first results were given in \cite{Bourgain:1996} by Bourgain, who proved that for any $\suca$ there exists $p=p(\suca)$  such that small initial data 
in the $H^{p'+p}$ norm stay small  in the $H^{p'}$ norm, for times of order $\delta^{-\suca}$. 
Afterwards, Bambusi in \cite{Bambu:1999b} proved that superanalytic initial data stay small in analytic norm for subexponentially long times.
Following the strategy proposed in \cite{Bambusi:2003} for the Klein-Gordon equation Bambusi and Gr\'ebert in \cite{Bambusi-Grebert:2003} first considered  NLS equations  on  $\T^d$  and then, in   \cite{Bambusi-Grebert:2006}, proved 
polynomial bounds for a class of {\it tame-modulus} PDEs.
Similar results were also proved for the Klein Gordon equation on tori and Zoll manifolds in \cite{DS0},\cite{DS},\cite{BDGS}. Successively Faou and Gr\'ebert  in \cite{Faou-Grebert:2013} considered the case of analytic initial data and proved subexponential bounds on the stability time
for classes of NLS equations in $\T^d$. In \cite{BMP18} the first author with Biasco and Massetti studied an abstract Birkhoff normal form on {\it sequence spaces} proving subexponential stability times for Gevrey regular initial data. A similar result was proved in \cite{Cong}. 
 An interesting feature of the last three papers is that instead on relying on tameness properties they use the fact that  the equations they study have some symmetries, namely they are gauge and translation invariant (actually in \cite{BMP18} the translation invariance condition is weakened).
\\
All the preceding results regard semilinear PDEs.
Regarding equations with derivatives in the nonlinearity, the first results were in \cite{Yuan-Zhang} for the semilinear case. Then we mention \cite{Delort-2009,DS15}  for the Klein-Gordon equation, \cite{Berti-Delort} for the water waves and \cite{FI} for the reversible NLS equation. Recently, Feola and Iandoli, \cite{FI2} prove polynomial lower bounds for the stability times of Hamiltonian NLS equations with two derivatives in the nonlinearity. In the context of infinite chains with a finite range coupling, similar considerations can be done and we mention \cite{BenFroGior}.
\subsection{Statements}
We study  Hamiltonians on infinite dimensional sequence spaces, which are higher  order ({\it $M$-regular}) analytic perturbations of quadratic Hamiltonians nearby an elliptic  fixed point (i.e a zero) and satisfying the {\it Momentum conservation} property, namely they are formally translation invariant,  see Definition \ref{mconserv}. 
\\
 We first show that the space $\cF$  of formal Hamiltonians in infinite variables $u=\pa{u_j}_{j\in \Z}$ satisfying this {\it Momentum conservation property} is well defined and  closed w.r.t Poisson brackets, then we define a scaling degree (which is the homogeneity degree minus two, see Definition \ref{scialla}, so that the degree of the Poisson bracket of two functions is the sum of the respective degrees) so that $\cF$ has a natural filtered Lie algebra structure. Thus $\cF$ is decomposed in homogeneous components $\cF^d$  and we define  $\cF^{\ge d}:= \widehat{\oplus}_{h\ge d}\cF^{h}$.
 \\
Given a  rationally independent $\omega\in \R^\Z$,  namely such that all non-trivial finite rational combinations of $\omega$ are non zero, we consider the affine space $D_\omega +\cF^{\ge 1}$ of formal Hamiltonians of the form 
 \begin{equation}
 \label{lequazione}
 H= D_\omega + P \,,\quad D_\omega= \sum_{j\in \Z} \omega_j |u_j|^2\,,\qquad P= O(u^3)\,,
 \end{equation}
  and acting on this space we define the group of formal symplectic (i.e canonical) transformations $e^{\set{\cF^{\ge 1},\cdot}}$.
  Finally we define the space of normal forms as those formal Hamiltonians which Poisson commute with $D_\omega$. We prove the following
  \begin{thm*}
All  Hamiltonians $H$ as above  are formally symplectically conjugated  to normal form. Moreover the normal form Hamiltonian associated to $H$ is unique.
 \end{thm*}
Having properly developed the formal framework, we consider the question of formal vs. analytic linearization in the infinite dimensional setting on the phase space of Gevrey regular functions.
\\
In order to keep technical difficulties to a minimum, we work on
Nonlinear Schr\"odinger like Hamiltonians of the form
with the standard symplectic structure on $\ell_2= \ell_2(\Z,\C)$. As phase space we consider the sequences of {\it Gevrey}  regularity, namely
we  consider the weighted space
\begin{equation}
\label{gevrey}
\th_{s,p,\theta} :=\set{ u\in \ell^2(\Z,\C): \quad |u|^2_s:=\sum_{j\in\Z}\jap{j}^{2p}e^{2s\jap{j}^\theta}|u_j|^2<\infty   }
\end{equation}
where $\jap{j}:=\max(|j|,1)$, $s>0$, $p\geq \frac12$ and $0<\theta< 1$. Then, given $r>0$, we consider the space of {\it M-regular Hamiltonians} $P\in \cH_r(\th_{s,p,\theta})$, such that
the Cauchy majorant of the map $u\to X_P(u)$ is analytic from the ball $B_r (\th_{s,p,\theta})$, centered at the origin and of radius $r$ into $\th_{s,p,\theta}$.
\\
Now we consider a Hamiltonian as in \eqref{lequazione}, with the additional condition that $P\in \cH_{r_0}(\th_{s_0,p,\theta})$ and the 
 frequency $\omega$ is ``Diophantine" in the following  sense introduced by Bourgain \cite{Bourgain:2005}.
 We set
 \begin{equation}
 \Omega_{}:=\set{\omega=\pa{\omega_j}_{j\in \Z}\in \R^\Z,\quad \sup_j|\omega_j-j^2| < 1/2 }
 \end{equation}
 \begin{defn} Given $\gamma>0$ , we  denote by $\dg$ the set  of \sl{Diophantine} frequencies 
 	\begin{equation}\label{diofantinoBIS}
 	\dg:=\set{\omega\in \Omega_{}\,:\;	|\omega\cdot \ell|> \gamma \prod_{n\in \Z}\frac{1}{(1+|\ell_n|^{2} \jap{n}^{2})}\,,\quad \forall \ell\in \Z^\Z_f\setminus \{0\}}.
 	\end{equation} 
 \end{defn}
 The map $\pa{\omega_j}_{j\in\Z}\to  \pa{j^2-\omega_j}_{j\in\Z}$ identifies $\Omega$ with $[-1/2,1/2]^\Z$. Hence we  endow $\Omega$ with the product topology and with the corresponding  probability measure. With respect to such measure Diophantine frequencies are typical, namely $\Omega\setminus\tD_\g$ has measure {proportionally bounded by $\g$ (see \cite{BMP18}[Lemma 4.1])}.
 \\
  Then  we prove~:
\begin{thm*}
 	{\it If $H$ is formally conjugated to $D_\omega$, then there exists $r_1<r_0$,  $s_1> s_0$ and a close to identity analytic symplectic change of variables 
 $
 \Psi: B_{r_1}(\th_{s_1,p,\theta}) \to \th_{s_1,p,\theta}
 $
 such that $H\circ \Psi = \sum_{j\in \Z} \omega_j |u_j|^2$.}
\end{thm*}
\subsection{Strategy}
In order to describe our strategy consider a finite dimensional Hamiltonian system with a non-degenerate  elliptic fixed point, which in the standard complex symplectic coordinates
$u_j= \frac{1}{\sqrt2}(q_j+ \im p_j)$ 
is described by the Hamiltonian 
\begin{equation}
\label{finito}
H= \sum_{j=1}^n \omega_j |u_j|^2 + O(u^3)\,,\quad \mbox{where $\omega_j\in \R$ are the {\it linear frequencies}}.
\end{equation}
Here if the frequencies $\omega$ are rationally independent, then one can perform the so-called Birkhoff normal form procedure: for $\suca\geq 1$ 
Hamiltonian \eqref{finito} is transformed into 
\begin{equation}
\label{finitobirk}
\sum_{j=1}^n \omega_j |u_j|^2 +Z +R\,
\,,
\end{equation}
where $Z$ depends only on the actions $(|u_i|^2)_{i=1}^n$ while $ R= O(|u|^{\suca +3})$ has a zero of order at least $\suca + 3$ in $\abs{u}$. 
At each step, the generating function of the change of variables is a polynomial, so it is analytic and generates a flow in a sufficiently small ball $B_\delta$ around the origin.
It is well known that this procedure generically diverges in $\suca$, but assuming that  $\omega$ is appropriately non resonant, say diophantine\footnote{A vector $\omega\in\R^n$ is called diophantine when it is badly approximated by rationals, i.e. it satisfies, for some $\gamma,\tau>0$, $\abs{k\cdot \omega} \ge \gamma \abs{k}^{-\tau},\quad \forall k\in\Z^n\setminus \set{0}\,$.} one can control $R$ and hence find  $\suca=\suca(\delta)$ which minimizes the size of the remainder $R$. It can be shown that it is bounded by an exponentially flat function of $\delta$, of order related to $\tau$ (for a general treatment, see instance, \cite{lombardi-nf,stolo-lombardi}). This phenomenon is also related to Nekhoroshev kind of result \cite{Posch-Nekho,BGG85,Nekho,Nid,BCG}. 
\\
If $H$ in \eqref{finito} is "formally linearizable", namely there exists a formal symplectic change of variables which conjugates $H$ to $\sum_{j=1}^n \omega_j |u_j|^2$, and $\omega$ is Diophantine, then at each step of the procedure described above, \bu{we find} $Z=0$ and one can prove convergence.
In order to apply this general scheme in the infinite dimensional setting we first discuss the BNF procedure at the level of  formal power series. Here the fundamental difference w.r.t. the finite dimensional case is that even polynomials can be just formal power series, so it is not a priori obvious that the space 
of  formal power series is well defined and has a Poisson algebra structure (which coincides with the usual one on finite dimensional subspaces).
As a simple example consider the formal  power series $ H= \sum_j u_j$, then 
\[
\{H,\bar H\}= \sum_i\sum_j \{u_j\,,\bar u_i\}=\infty\,.
\]
 We show  that for translation invariant formal Hamiltonians the Poisson brackets are well defined (see also \cite{FGP} ), and that formal Hamiltonians are a filtered Lie algebra with respect to a {\it scaling degree}.  Then we define a group of formal symplectic changes of variables, and prove our BNF result. In order to define our changes of variables and prove the group structure we strongly rely on the properties of the {\it scaling degree} as well as on the Baker Campbell Hausdorf formula.
 \\
 Then we restrict to functions on the sequence space $\th_{s,p,\theta}$, introduce the space of regular Hamiltonians and state the main relevant properties. All properties were proved in \cite{BMP18} in the more restrictive case of Gauge invariant Hamiltonians, so we follow the same strategy; for completeness we give all the proofs in the appendix.
 One we have all the basic properties needed to perform Birkhoff Normal Form, proving that formal linearizability implies analytic linearizability becomes a relatively straightforward induction.


\section{ Formal Birkhoff  Normal on sequence spaces}
As usual given a vector $k\in \Z^\Z$, $|k|:=\sum_{j\in\Z}|k_j|$. We denote $\N^\Z_f$ to be the set of  finitely supported sequences of non negative integers, similarly for $\Z^\Z_f$. If $j\in\mathbb{Z}$ then $\be_j\in \Z^\Z_f$ denotes the vector the $j$-coordinate of which is  $1$, while the others are zero.
\begin{defn}[Formal power series]\label{Hr}
	We consider the space $\cF$ of formal  power series expansions  
 	in $u\in \C^\Z$:
	$$ 
	H(u)  = \sum_{\substack{\bal,\bbt\in\N^\Z_f} }H_{\bal,\bbt}u^\bal \bar u^\bbt\,,
	\qquad u\in \C^\Z,\quad
	u^\bal:=\prod_{j\in\Z}u_j^{\bal_j} \quad |v|: = \sum_i|v_i|
	$$ 
	with the following properties: 
	\begin{enumerate}
		\item $H_{0,0}= 0$, $H_{\be_0,0}= H_{0,\be_0}=0$ 
		\item Reality condition:
		\begin{equation}\label{real}
			H_{\bal,\bbt}= \overline{ H}_{\bbt,\bal}\,;
		\end{equation}
		\item Momentum conservation:
		\begin{equation}
			H_{\bal,\bbt}= 0 \quad\mbox{if}\;\, \pi(\bal,\bbt):= \sum_{j\in\Z} j(\bal_j-\bbt_j)\ne 0 \label{mconserv}
		\end{equation}
	\end{enumerate}
\begin{rmk}
	The condition (3) means that the formal Hamiltonian is invariant w.r.t. the symmetry  $u_j\to e^{  \im j\tau} u_j$, $\tau\in\mathbb{R}$.
\end{rmk}

We shall denote 
\[
\cM:=\{(\bal,\bbt) \in \N^\Z_f:\qquad \pi(\bal,\bbt)=0\}
\]
so that $H\in \cF$ can be written as
\[
\sum_{(\bal,\bbt)\in \cM}H_{\bal,\bbt}u^\bal \bar u^\bbt
\] 

	Finally we define 
	\begin{equation}\label{nocciolina}
		\cK:=\left\{Z\in \cF\, : \, Z(u) = \!\!\sum_{\substack{\bal\in \N^\Z_f}} Z_{\bal,\bal}|u|^{2\bal}\right\}\,,\qquad \cR:=\left\{R\in \cF\, : \, R(u) = \!\!\!\!\!\!\!\!\sum_{\bal,\bbt\in \cM:\, \bal\ne \bbt}\!\!\!\!\!\!\!\! R_{\bal,\bbt}u^\bal\bar u^\bbt\right\}
	\end{equation}
	and we can decompose $\cF= \cK\oplus \cR$  as each element of $\cF$ can uniquely be expressed in term of monomials the coefficients of which is either zero or not zero. 
\end{defn}
\begin{defn}[ scaling degree]\label{scialla}
	For $d\in \N$, we denote by $\cF^d\subset \cF$ the vector space of homogeneous formal polynomials of degree $\td+2$, and define 
\[
	\cF^{\le d}= \oplus_{h\le d} \cF^{h}\,,\quad \cF^{>d}:=\widehat\oplus_{h>d} \cF^{h} \,,\quad  \cF^{\ge d}:= \cF^{>d}\oplus \cF^{d}\,,\cF =\cF^{\le d}\oplus\cF^{> d},\dots
	\]
		We define the projections  associated to these direct sum decompositions
	\[
	\Pi^{(\td)} H= \sum_{|\al|+|\bt|=\td+2}H_{\al,\bt} u^\al\bar u^\bt \,,\quad \Pi^{(> \td)} H= \sum_{|\al|+|\bt|> \td+2}H_{\al,\bt} u^\al\bar u^\bt\,,\dots
	\]
	Elements of $\cF^{\geq d}$ (resp. $\cF^{> d}$) are said to be {\it of scaling order} $\geq d+2$ (resp. $>d+2$). In the sequel, for simplicity, we shall just say that an element of $f\in \cF^{\geq d}$ is of "order d" and we shall say that $f$ is "exactly of order d" if it has a non vanishing component $\Pi^{(\td))}f$ in $\cF^d$.
	Finally we define
	\[
	\Pi^\cK H =  \sum_{\al}H_{\al,\al} |u|^{2\al}\,,\quad \Pi^\cR H = \sum_{\al\ne \bt}H_{\al,\bt} u^\al\bar u^\bt.
	\]
	We denote by $\cK^d:= \cF^d \cap \cK$  and similarly for $\cR$ and $\ge \td ,\le \td$.
	Note that $\cF= \widehat \oplus_d \cF^d$.
\end{defn}
\begin{rmk}
Of course, since we are in infinite dimension, even if  the $\cF^d$ are homogeneous  they are  only formal polynomials.  However if we restrict to monomials $u^\bal\bar{u}^\bbt$ with $|\bal_j|+|\bbt_j| =0$ for all $j>N$  we are working on the usual space of polynomials on which we have the standard symplectic structure $\im \sum_{j\le N} d u_j\wedge d \bar u_j$. We now  show that such structure extends to $\cF$.
\end{rmk}
	\begin{prop}\label{degree decompositionbis}
		The following Formula \eqref{well} is well defined and  endows $\cF$ with a  Poisson algebra structure  which is a filtered  Lie algebra w.r.t. the $\cF^{\geq d}$'s. 
		\begin{equation}\label{well}
			\set{F,G} := \im\sum_{  (\bal^{(i)},\bbt^{(i))}\in\cM }  \bcoeffu{F}\bcoeffd{G} \sum_j \pa{\baluno_j \bbtdue_j  - \bbtuno_j \baldue_j }u^{\baluno + \baldue -\be_j}\bar{u}^{\bbtuno + \bbtdue - \be_j}
		\end{equation}
		\end{prop}
	Before proving our assertion we need a technical lemma.
	Let $\be_j\in \N^\Z_f$ be the $j$th vector of the standard basis.
	\begin{lemma}\label{finmo} 1)\quad  Given $\alpha\in \N^\Z_f$ there  is only a finite number of pairs    $ \alpha^{(1)},\alpha^{(2)}\in \N^\Z_f$ with $\alpha=\alpha^{(1)}+\alpha^{(2)}$.
		2)\quad 	
		Given  $ (\bal,\bbt)\in \cM$ there is only a finite number of pairs    $ (\baluno,\bbtuno),\  (\baldue,\bbtdue)\in \cM$ and indices $j\in\Z$  such that:
		\begin{enumerate} 
			\item[\it i)] $ (\bal,\bbt)= (\baluno,\bbtuno)+(\baldue,\bbtdue)-(\be_j,\be_j)$ 
			\item[\it ii)] one has $\bal_j^{(1)}\bbt_j^{(2)}+ \bal_j^{(2)}\bbt_j^{(1)}\ne 0$.
		\end{enumerate}  
	\end{lemma} 
	\begin{proof} 1) is clear since for all $j$ one has $0\le (\alpha_1)_j\leq \alpha_j  $.\smallskip
		\\
		2)\quad By item 1) we may divide $(\bal,\bbt)= (a^{(1)}, b^{(1)})+ (a^{(2)}, b^{(2)})$ in a finite number of ways. Then 
the pairs 	$(\baluno,\bbtuno),(\baldue,\bbtdue)$  can only have one of the following forms (up to exchanging the indices)

A)	$\quad (\baluno,\bbtuno)= (a^{(1)}, b^{(1)})+ (\be_j,\be_j)\,,\quad (\baldue,\bbtdue) = (a^{(2)}, b^{(2)})$

B) $\quad (\baluno,\bbtuno)= (a^{(1)}, b^{(1)})+ (\be_j,0)\,,\quad (\baldue,\bbtdue) = (a^{(2)}, b^{(2)})+(0,\be_j)$,
\\
for some index $j\in \Z$.
\\
If we are in case $A)$ then 	by condition $ii)$  we have $j\in $Supp$(a^{(2)}+b^{(2)})$, which restricts to a finite number of possible $j's$. Otherwise in case $B)$ 
	by momentum conservation e have $j= -\pi( a^{(1)}, b^{(1)})={\pi} (a^{(2)}, b^{(2)})$ and again $j$ is restricted to a finite number of possible choices.
	\end{proof}
	\begin{proof}[Proof of Proposition \ref{degree decompositionbis}]
		The fact that the Poisson bracket is well defined  follows immediately from the previous Lemma. 
		Indeed by construction
			\[
		\{F,G\}= \sum_{\bal,\bbt} P_{\bal,\bbt} \buu \in \cF
		\]
		where  $P_{\bal,\bbt}=0$ if $\pi(\bal,\bbt)\ne 0$ and otherwise
		\begin{equation}
			\label{begon}
			P_{\bal,\bbt}= {\im}\sum_j \sum_{\substack {\bal^{(i)},\bbt^{(i)}\in\N^\Z_f:\; \pi(\bal^{(i)},\bbt^{(i)})=0\\ \bal = \baluno+\baldue - \be_j\,, \bbt = \bbtuno+\bbtdue - \be_j }}
			\bcoeffu{F}\bcoeffd{G} \pa{\baluno_j \bbtdue_j  - \bbtuno_j \baldue_j }\,.
		\end{equation}
	Then item 2 of the previous Lemma  implies that $P_{\al,\bt}$ above is given by a finite sum.
	\\
		The fact that it  endows $\cF$ with a  Poisson algebra structure follows from the fact that   the infinitely many identities defining such a structure   involve  only a finite number of elements $u_i,\bar u_i$  and then we are in the canonical  Poisson algebra.
		\\
The filtered Lie algebra property comes from the fact that in \eqref{begon} we get $|\alpha|+|\beta| = |\baluno|+ |\baldue|+ |\bbtuno|+ |\bbtdue| -2$, this shows that if $F\in \cF^{\ge \td_1}$, and $G\in \cF^{\ge \td_2}$  then 
		\begin{equation}
			\label{a+b}|\alpha|+|\beta| \ge \td_1+2+\td_2+2-2 = \td_1+\td_2+2.
		\end{equation}  so 
	$\{F,G\}\in \cF^{\ge \td_1+\td_2}$.
	\end{proof}

\begin{rmk}\label{gogna}
Let $H_i\in \cF^{\ge \td_i}$ be a sequence of formal Hamiltonians with  $\td_{i+1}\ge \td_i$ for all $i\ge 1$. Then the series
\[
H= \sum_{i=1}^\infty H_i \in \cF^{\ge \td_1}
\]
is well defined since for any $\td\ge \td_0$ the projection \[\Pi^{(\td)} H = \Pi^{(\td)} \sum_{i: \td_i \le \td}  H_i  \] is a finite sum.
\end{rmk}
	We say that  a linear operator $L:\cF\to \cF$ is of order (or increase the order by) $\td$  if for all $h$ 
	\[
	L:\cF^{ \ge h} \to \cF^{ \ge h+\td}\,.
	\]

\begin{lemma}\label{lemniscata}
let $L_n$ be a sequence of linear operators on $\cF$  and let  $\td_n$ be the order of $L_n$. If the sequence $\td_n$ increases to infinity then 
\[
L:= \sum_{n=1}^\infty L_n \,,\qquad T =\prod_{n=1}^\infty (\id + L_n)-\id 
\]
are  linear operators on $\cF$ of order  $\td_1$.
\end{lemma}
\begin{proof}
For the first statement, for all $\td\in \N$  let   $N(\td)$  be the largest $N$ such that $\td_N\le \td$. By construction
$\Pi^{(\le \td)} L_n K=0$ for all $n> N(\td)$ and for any $K\in \cF$.
 Then for all $K\in \cF$ and $N>N(\td)$ one has
	\[
\Pi^{(\le \td)}	\sum_{n=1}^N L_n K  = \Pi^{(\le \td)}	\sum_{n=1}^{N(\td)} L_n K\,,
	\] and the claim follows.
	\\
	Regarding the second statement we proceed similarly
	\[
		\prod _{n=1}^N (\id + L_n ) =  \prod _{n=1}^{N-1} (\id + L_n ) +L_N \prod _{n=1}^{N-1} (\id + L_n )\,,
	\] hence, {for all $\td\geq 0$ and all $N> N(\td)$}
	\[
\Pi^{(\le \td)}\prod _{n=1}^N (\id + L_n ) =  \Pi^{(\le \td)}\prod _{n=1}^{N(\td)} (\id + L_n )\,.
\]
\end{proof} As a direct consequence we have the following.
\begin{cor}
	Given $G\in \cF^{\ge  \td}$, with $\td\ge 1$  we define 
	\begin{equation}
		\label{flusso}
	\ad_G:= \{G,\cdot \} \,,\quad \Phi_G:=	\exp(\{G,\cdot\})= \sum_{k\ge 0} \frac{\ad_G^k}{k!}\,,
	\end{equation}
	then $ 	\ad_G$ and  $\Phi_G-\id$  are operators of order $\td$, namely
	\[
	\ad_G,\Phi_G -\id: \cF^{\ge h} \to \cF^{\ge h+\td}\,.
	\]
	Similarly for any sequence $b_k$  one has that 
	\[
	\sum_{k\ge \tn } b_k{\ad_G^k}: \cF^{\ge h} \to \cF^{\ge h+\td\tn }\,.
	\] 
	\end{cor}
\begin{defn}
	Given $G\in \cF^{\ge 1}$ we call the operator $\Phi_G$ defined in \eqref{flusso} a formal symplectic change of variables on $\cF$.
\end{defn}
The following Lemma ensures the group structure of the formal symplectic changes of variables
\begin{lemma}[Baker-Campbell-Haussdorf]\label{BCH}
Given   $F\in \cF^{\ge \td_1}$ and  $G\in \cF^{\geq \td_2}$, with $\td_i\ge 1$,  then there exists $K\in \cF^{\ge 1}$, such that
	\[
	e^{\{G,\}}e^{\{F,\}}=  e^{\{ K,\}} \,,\qquad  K-F-G \in  \cF^{ \ge \td_1+\td_2}
	\] 
\end{lemma}
\begin{proof}
	By the Baker-Campbell-Hausdorff formula (\cite{serre-Lie1}[p.29]) one has
	\[
K:= 
	\sum_{n=1}^\infty \frac{(-1)^{n-1}}{n} \sum_{r_i+s_i>0} \frac{[G^{r_1}F^{s_1}\dots  G^{r_n}F^{s_n}]}{
		\left(\sum_{i=1}^n(r_i+s_i)\right) \prod_{i=1}^n r_i! s_i!}
	\]
	where
\begin{equation}
		\label{hoho}
	[G^{r_1}F^{s_1}\dots G^{r_n}F^{s_n}]:= \begin{cases}
		\ad_G^{r_1}\ad_F^{s_1}\dots \ad_G^{r_n}F  \quad \mbox{ if} \; s_n=1
		\\
		\ad_G^{r_1}\ad_F^{s_1}\dots \ad_F^{s_{n-1}} G  \quad \mbox{ if}\; s_n=0\,,\; \mbox {and} \; r_n=1
		\\
		0  \quad \quad \quad \quad  \mbox {otherwise}
	\end{cases}
\end{equation}
	Recalling that   $F\in \cF^{\ge \td_1}$ and  $G\in \cF^{\geq \td_2}$, each term $\ad_G^{r_1}\ad_F^{s_1}\dots \ad_G^{r_n}F$ (resp. $\ad_G^{r_1}\ad_F^{s_1}\dots \ad_F^{s_{n-1}}G$)
	is of order $\left(\sum_{i=1}^{n}r_i\right)\td_2+(\sum_{i=1}^{n}s_i)\td_1\geq n \min(\td_1,\td_2)$. 
	Hence setting $N(\td)$ to be the largest $N$ such that $ N \min(\td_1,\td_2) \le \td$n 
	\[
	\Pi^{\le \td} K= \Pi^{\le \td} \sum_{n=1}^{N(\td)} \frac{(-1)^{n-1}}{n} \sum_{r_i+s_i>0} \frac{[G^{r_1}F^{s_1}\dots  G^{r_n}F^{s_n}]}{
		\left(\sum_{i=1}^n(r_i+s_i)\right) \prod_{i=1}^n r_i! s_i!}
	\]
	\\
	Moreover if $n\ge 2 $ then  the Hamiltonian in \eqref{hoho} is of order $\ge \td_1+\td_2$, so $K-F-G \in  \cF^{ \ge \td_1+\td_2}$. 
\end{proof}
\begin{lemma} \label{compo} Given a sequence of generating functions $G_i\in \cF^{{\geq} \td_i}$ with $\td_{i+1} > \td_i\ge 1$ then there exists $\cG\in \cF^{\ge d_1}$ such that the composition
	\[
\prod_i	e^{\{G_i,\}} = e^{\{\cG,\}}
	\] 
	\end{lemma}
\begin{proof} 
By Lemma \ref{lemniscata} with $L_n = e^{\{G_n,\}}  -\id$ we know that  $\prod_i	e^{\{G_i,\}}$ is a well defined operator of $\cF$. 
Using Lemma \ref{BCH} we can define $F_k\in \cF^{\ge1}$ iteratively so that
	\[
 e^{\{F_k,\cdot\}}= e^{\{G_k,\cdot\}} e^{\{F_{k-1},\cdot\}}
	\]
	since $e^{\{G_k,\cdot\}} -\id$ is of order $\td_k$  there exists $N(\td)$ such that if $k>N(\td)$ then 
	\[
\Pi^{(\le \td )}	F_k = \Pi^{(\le \td )}	F_{N(\td)}
	\]
Then  $\cG= \lim_{k\to \infty} F_k$ is well defined.
\end{proof}
For any vector $\omega\in \R^\Z$ such that
\[\omega\cdot\ell \ne 0\,,\quad \forall \ell\in \Z^\Z_f\setminus\{0\}\,,
\] 
we define the {\it non-resonant} quadratic Hamiltonian
\[
D_\omega:= \sum_j \omega_j |u_j|^2\,.
\]
\begin{lemma}\label{uff}
	The operator $\ad_{D_\omega}$ is invertible on $\cR^{(\td)}$ for all $\td$.
\end{lemma}
\begin{proof}
	Given $F\in \cR^{(\td)}$,  we have
	\begin{equation*}
	\set{D_{\omega},G} = \im\sum_{\substack {\bal^{(2)},\bbt^{(2)}\in\N^\Z\,, \\
			|\bal^{(2)}|+|\bbt^{(2)}|<\infty\,,\, \pi(\bal^{(2)},\bbt^{(2)})=0}}  \bcoeffd{G} \left(\sum_j \omega_j\pa{\bbtdue_j  - \baldue_j }\right)u^{\baldue}\bar{u}^{\bbtdue}=F
	\end{equation*}
	Hence, we have $G:=\ad_{D_\omega}^{-1}(F)$ with for all $\bal^{(2)},\bbt^{(2)}\in\N^\Z$, $\bal^{(2)}\neq\bbt^{(2)}$ with $|\bal^{(2)}|+|\bbt^{(2)}|<\infty$ and $\pi(\bal^{(2)},\bbt^{(2)})=0$,  
	\[
	\bcoeffd{G}:=\bcoeffd{F} \left(\sum_j {\im}\omega_j\pa{\bbtdue_j  - \baldue_j }\right)^{-1},\quad G_{\alpha^2,\alpha^2}=0.
	\]
\end{proof}
\begin{prop}[Birkhoff Normal Form] \label{giova}Given any formal Hamiltonian of the form
\begin{equation}
	\label{Birkhof}
	H= D_\omega + Z + R
\end{equation}
where $Z\in \cK^{\ge 2}$ and $R\in \cF^{\ge \td}$  with $\td\ge 1$, then
	 \begin{enumerate}
		\item {\bf Formal Normal Form:} there exists $S\in  \cF^{\ge \td}$ such that
	\[
	e^{\{S,\cdot\}} H = D_\omega  + \wtZ\,,\quad \wtZ- Z \in \cK^{\ge \td}\,.
	\]
	\item {\bf Uniqueness:} if $G\in \cF^{\ge 1}$ is such that $e^{\{G,\cdot\}} H \in \cK$ then $e^{\{G,\cdot\}} H= e^{\{S,\cdot\}} H$. Hence to each $H$ as above we can associate a unique  $Z_H\in \cK^{\ge 2}$ such that $e^{\{S,\cdot\}} H= D_\omega+ Z_H$.
	\end{enumerate}
\end{prop}
\begin{proof}
For item $(1)$ Let us first consider the case $\td\ge 2$. we start with a Hamiltonian $H_0\in \cF$ of the form $D_\omega+Z_0+  P_0$ with $P_0\in \cF^{\ge d}$ and we iteratively construct a sequence of generating functions
$S_i \in  \cR^{\ge 2i+d}$ and Hamiltonians $H_i$  by setting 
\[
\{D_\omega,S_i\} =\Pi^{\cR}  H_{i}\,,\quad H_{i+1}= e^{\{S_i,\cdot\} }H_{i}\,.
\]
We now show inductively that  for each $i$
\[
\Pi^{(< 2i +d)} \Pi^{\cR} H_i =0 \,, \quad S_i \in  \cR^{\ge 2i+ d}
\]
so in other words
\[
H_i = D_\omega+ Z_i + P_i\,,\quad Z_i \in \cK\cap \cF^{\le 2i +d-1} \,,\quad P_i\in \cF^{\ge 2i+d} \,.
\]
For $i=0$ we just set $Z_0= \Pi^{< d} Z$ and $P_0= R+ \Pi^{\ge d} Z$.
By induction we assume that $P_i \in \cF^{\ge 2i+d}$. Then by  Lemma \ref{uff}, $S_i \in \cF^{\ge 2i +d}$.
\begin{align*}
	e^{\{S_i,\cdot\} }H_{i} &=D_\omega+ Z_i + P_i + \{S_i, D_\omega\} + \sum_{h=2}^\infty \frac{\ad_{S_i}^{h-1}}{h!} \{S_i,D_\omega \}
	+ \sum_{k=1}^\infty \frac{\ad_{S_i}^k}{k!}(Z_i + P_i) \\
	&= D_\omega+ Z_i + \Pi^\cK P_i  - \sum_{k=1}^\infty \frac{\ad_{S_i}^{k}}{(k+1)!}\Pi^{\cR}  P_{i}  + \sum_{k=1}^\infty \frac{\ad_{S_i}^k}{k!}(Z_i + P_i)\,.
\end{align*}
So we may set
\[
Z_{i+1}:= Z_i+
\Pi^{(< 2i +d+2)}\Pi^\cK P_i \,,\quad P_{i+1}= e^{\{S_i,\cdot\} }H_{i}- D_\omega - Z_{i+1} 
\]
and verify that $P_{i+1}\in \cF^{\ge 2i+d+2}$ by applying Proposition \ref{degree decompositionbis} and noticing that, since $4i +2\td \ge 2i +\td+2$, the term of lowest degree is $\{S_i,Z\}$.
\\
 Then
we set
\[
\widetilde Z= \lim_{i\to \infty} Z_i = Z_0+ \sum_{i=0}^\infty \Pi^{\le 2i +d+1}\Pi^\cK P_i = Z_0+ \Pi^\cK \sum_{i=0}^\infty \Pi^{(< 2i +d+2)} \Pi^{(\ge 2i +d) }P_i\,,
\] which is well defined by Remark \ref{gogna}. Finally by Lemma \ref{compo} we can define $S\in \cF^{\ge \td}$ so that
\[
e^{\{S,\cdot\}}= \prod_{i=0}^\infty e^{\{S_i,\cdot\}}\,.
\]
If $\td=1$ we perform a preliminary step in order to increase the degree by one and then we start the procedure explained above. 
We start with $H = D_\omega+ P$, with $P:= R+ Z$ .  As before we fix $S\in \cR^{\ge 1}$ so that
$\{D_\omega ,S\}=\Pi^\cR H$ we set
\begin{align*}
H_0:=	e^{\{S,\cdot\} }H = D_\omega+ \Pi^\cK P  - \sum_{k=1}^\infty \frac{\ad_{S}^{k}}{(k+1)!}\Pi^{\cR}  P  + \sum_{k=1}^\infty \frac{\ad_{S}^k}{k!}P\,.
\end{align*}
then fixing $Z_0:= \Pi^{\le 2} \Pi^\cK P$ and $P_0:= H_0-D_\omega-Z_0$ we are in the setting of the previous case.


Regarding item $(2)$ we remark that If $e^{\{S_1,\cdot\}}$ transforms a normal form $D_{\om}+K_1$ into a normal form  $D_{\om}+K_2$, then 
$$
e^{\{S_1,\cdot\}}(D_{\omega}+K_1)= D_{\omega}+K_1+\sum_{h=1}^\infty \frac{\ad_{S_1}^{h-1}}{h!} \{S_1,D_\omega+K_1 \}=D_{\omega}+K_2.
$$
Since $\cK=\cK^{\ge 2}$ and $S\in \cH^{\ge 1}$, comparing homogeneous terms of  degree $1$ we get $\{S_1,D_\omega\}=0$ so we should have $S_1^{(1)}\in\cK$ which can only be possible if $S_1^{(1)}=0$.
Comparing homogeneous terms of degree $2$, we obtain $K_1^{(2)}-K_2^{(2)}+\{S_1^{(2)}, D_{\om}\}=0$. Recalling that $\{S_1^{(2)}, D_{\om}\}\in \cR$ we have  $K_1^{(2)}-K_2^{(2)}\in \cK\cap\cR$ is zero and $S_1^{(2)}\in\cK$. Assuming that $K_1^{(j)}=K_2^{(j)}\in \cK$ and $S_1^{(j)}\in\cK$ for $2\leq j\leq m$. Then we have
\begin{eqnarray*}
K_1^{(m+1)}-K_2^{(m+1)}&+&\{S_1^{(m+1)}, D_{\om}\}+\sum_{h=2}^\infty \frac{1}{h!}\left(\sum_{j_1+\cdots+j_l=m+1}\{S_1^{(j_1)},\{S_1^{(j_2)},\cdots \{S_1^{(j_h)},D_\omega\}\}\}\right.\\
&+&\left.\sum_{j_1+\cdots+j_h+j_{h+1}=m+1}\{S_1^{(j_1)},\{S_1^{(j_2)},\cdots \{S_1^{(j_h)},K_1^{(j_{h+1})} \}\}\}\right)=0
\end{eqnarray*}
By induction and since $D_{\om}$ is non resonant, then both sums above are zero. Hence, we the same reasoning as above, we obtain $K_1^{(m+1)}=K_2^{(m+1)}\in \cK$ and $S_1^{(m+1)}\in\cK$. The result follows from \rp{compo}.

\end{proof}
\begin{cor}\label{ofelia}
For any $H$ as in (\ref{Birkhof}), if for $G\in \cF^{\ge 1}$ one has $e^{\{G,\cdot\}} H= D_\omega+ Z +R$ with $R\in \cF^{\ge d_1}$ then $Z- Z_H\in \cK^{\ge d_1}$.
\end{cor}
\begin{proof}
By Proposition \ref{giova} $(1)$ there exists $S\in \cF^{\ge d_1}$ which normalizes $D_\omega+ Z +R$ to $D_\omega+ \widetilde Z$ with $\widetilde Z-Z\in \cF^{\ge d_1}$. By Lemma \ref{compo} there exists $G_1\in \cF$ such that $e^{\{G_1,\cdot\}}=  e^{\{S,\cdot\}}e^{\{G,\cdot\}}$. Since $G_1$ puts $H$ in normal form, by Proposition \ref{giova} $(2)$, $\widetilde Z= Z_H$ and the result follows.
\end{proof}
\begin{defn}
	 We say that $H$ is formally linearizable if $Z_H=0$.
\end{defn}
\begin{cor}\label{keypoint}
If $H$ is formally linearizable and there exists a formal symplectic change of variables with $e^{\{S,\cdot\}} H= D_\omega+ Z +R$ with $R\in \cF^{\ge d}$ and $Z \in \cK^{< d}$ (this last condition does not imply any loss of generality) then $Z=0$.
\end{cor}
\begin{proof}
	This follows directly from Corollary \ref{ofelia}.
\end{proof}

If we know a priori that $H$ is formally linearizable then we get a faster growth of the degree of $P_i$.

\begin{lemma} \label{Birlinform}.
	If  $H_0\in \cF$ of the form $D_\omega+ P_0$ with $P_0\in \cF^{\ge 1}$ is formally linearizable then the sequence of generating functions 
\[
\{D_\omega,S_i\} =\Pi^{\cR}  H_{i}\,,\quad H_{i+1}= e^{\{S_i,\cdot\} }H_{i}\,.
\]
satisfies
\[
H_i = D_\omega + P_i \,,\quad P_i\in \cF^{\ge 2^i} \,.
\]
\end{lemma}
\begin{proof}
	By induction we assume that $P_i \in \cF^{\ge 2^i}$. Then by construction $S_i \in \cF^{\ge 2^i}$.
	\begin{align*}
	e^{\{S_i,\cdot\} }H_{i} &=D_\omega + P_i + \{S_i, D_\omega\} + \sum_{h=2}^\infty \frac{\ad_{S_i}^{h-1}}{h!} \{S_i,D_\omega \}
	+ \sum_{k=1}^\infty \frac{\ad_S^k}{k!} P_i \\
	&= D_\omega+  \Pi^\cK P_i  - \sum_{k=1}^\infty \frac{\ad_{S_i}^{k}}{(k+1)!}\Pi^{\cR}  P_{i}  + \sum_{k=1}^\infty \frac{\ad_{S_i}^k}{k!} P_i\,\\
	&=: D_\omega+ \Pi ^{<2^{i+1}} \Pi^\cK P_i + P_{i+1}\,.
	\end{align*}
	By Proposition \ref{degree decompositionbis} the two series in the formula above are in $\cF^{2^{i+1}}$ so to prove our claim we only need to show $\Pi^\cK \Pi ^{<2^{i+1}} P_i =0$. This is a consequence of Corollary \ref{keypoint}.
\end{proof}
\section{Regular Hamiltonians } 
We now revisit the formal Birkhoff normal form in the case of analytic Hamiltonians. We start by introducing an appropriate functional setting.
\subsection{Spaces of Hamiltonians}
Let us consider the weighted space
\[
\th_s= \th_{s,p,\theta} :=\set{ u\in \ell^2(\Z,\C): \quad |u|^2_s:=\sum_{j\in\Z}\jap{j}^{2p}e^{2s\jap{j}^\theta}|u_j|^2<\infty   }
\]
where $\jap{j}:=\max(|j|,1)$, $p\geq \frac12$ and $0<\theta\leq1$. The spaces $\th_{s,p,\theta}$ are contained in $\ell^2(\C)$, so we endow them with the  standard symplectic structure  coming from the Hermitian product on $\ell^2(\C)$. 
\\
We identify $\ell^2(\C)$ with $\ell^2(\R)\times \ell^2(\R)$ through $u_j= \pa{x_j+ i y_j}/\sqrt{2}$ and induce on $\ell^2(\C)$ the structure of a real symplectic Hilbert space\footnote{We recall that given a  complex Hilbert space $H$ with a Hermitian product $(\cdot,\cdot)$, its realification is a real symplectic Hilbert space with scalar product  and symplectic form given by
	\[
	\langle u,v\rangle = 2{\rm Re}(u,v)\,,\quad  \omega(u,v)= 2{\rm Im}(u,v)\,.
	\] } by setting, for any $(u^{(1)}, u^{(2)}) \in \ell^2(\C)\times \ell^2(\C)$,
\[
\langle u^{(1)},u^{(2)}\rangle = \sum_j \pa{x_j^{(1)}x_j^{(2)}+ y_j^{(1)}y_j^{(2)}} \,,\quad \omega(u^{(1)},u^{(2)})= \sum_j \pa{y_j^{(1)}x_j^{(2)}- x_j^{(1)}y_j^{(2)}},
\] 
which are the standard scalar product and symplectic form $\Omega= \sum_j dy_j\wedge d x_j$. \\

	Given $H\in \cF$, we define its majorant
as
\begin{equation}\label{betta}
\und H(u)  = \sum_{\substack{\bal,\bbt\in\N^\Z\,, \\
		|\bal|+|\bbt|<\infty} }|H_{\bal,\bbt}|u^\bal \bar u^\bbt\,.
\end{equation}
\begin{defn}[M-regular Hamiltonians]\label{Hreg}
	For $ r>0$,  let
	$\cH_{r,s}$
	be the subspace of $\cF$ of formal power series $H$ such that  $\und{H}$ is pointwise  absolutely convergent on 
	$B_r(\th_{s}) $, the ball of radius $r$ centered at the origin of $\th_s$, 
and 
$$
|H|_{ B_{r}(\th_{s})}
\equiv
\|H\|_{r,s}
:=
r^{-1} \pa{\sup_{\abs{u}_{\th_{s}}\leq r} 
	\abs{{X}_{{\underline H}}}_{\th_{s}} } < \infty\,.
$$
\end{defn}

\noindent

{\sl
Note that in $\cF$ one has $H(0)=0$ so this is actually a norm.
}

\smallskip

We shall show in the next subsection that $H\in \cH_{r,s}$  guarantees that   the Hamiltonian flow of $H$ exists at least locally and generates a symplectic transformation on $\th_s$, i.e. $\th_s$ is an invariant subspace for the dynamics. 

\begin{thm}[Main]
	Consider a Hamiltonian of the form
	\[
	\sum_{j\in \Z} \omega_j |u_j|^2 + P_0\,,\quad P_0\in \cH_{\tr,s_0} \cap \cF^{\ge 1}
	\]
	where $\omega\in 	 \dg$. Assume that there exists $G\in \cF^{\ge 1}$ such that 
	\[
	e^{\{G,.\}} H = \sum_{j\in \Z} \omega_j |u_j|^2\,,
	\]
	then there exists $r_1<\tr$,  $s_1> s_0$ and a close to identity change of variables $\Psi$ 
	\[
	\Psi: B_{r_1}(\th_{s_1}) \to \th_{s_1}
	\]
	such that $H\circ \Psi = \sum_{j\in \Z} \omega_j |u_j|^2$.
\end{thm}
\subsection{Poisson structure and homological equation}
The following Lemmata are proved in \cite{BMP18}  under the extra assumption of mass conservation, we discuss the proof in our slightly more general setting in the apendix.
\begin{lemma}\label{gasteropode}
	If $H\in \cH_{r,s}\cap \cF^{\ge \td}$, then for all $\rs\le r$ one has
	\begin{equation*}
	\norm{H}^{\wc}_{\rs,s} \le \pa{\frac{\rs}{r}}^{\td} \norm{H}^{\wc}_{r,s}\,.
	\end{equation*}
\end{lemma}
\begin{lemma}\label{genoveffa}
	If $H\in \cH_{r,s}$, then for all $s_1\ge s$ one has
	\begin{equation*}
	\norm{H}^{\wc}_{r,s_1} \le \norm{H}^{\wc}_{r,s}\,.
	\end{equation*}
\end{lemma}
\begin{lemma}[Poisson brakets and Hamiltonian flow]\label{ham flow}
		Let $0<\rho< r $,  and $F,G\in\cH_{r+\rho,\eta}(\th_{s})$, then 
		\begin{equation}\label{commXHK}
		\norm{\{F,G\}}_{r,s}
		\le 
		4\pa{1+\frac{r}{\rho}}
		\norm{F}_{r+\rho,s}
		\norm{G}_{r+\rho,s}\,.
		\end{equation}
For  $S\in\cH_{r+\rho,\eta}(\th_{s})$ with 
	\begin{equation}\label{stima generatrice}
	\norm{S}_{r+\rho,s} \leq\delta:= \frac{\rho}{8 e\pa{r+\rho}}. 
	\end{equation} 
	Then the time $1$-Hamiltonian flow 
	$\Psi^1_S: B_r(\th_{s})\to
	B_{r + \rho}(\th_{s})$  is well defined, analytic, symplectic with
	\begin{equation}
	\label{pollon}
	\sup_{u\in  B_r(\th_{s})} 	\norm{\Psi^1_S(u)-u}_{\th_{s}}
	\le
	(r+\rho)  \norm{S}_{r+\rho,s}
	\leq
	\frac{\rho}{8 e}.
	\end{equation}
	For any $H\in \cH_{r+\rho,s}$
	we have that
	$H\circ\Psi^1_S= e^{\set{S,\cdot}} H\in\cH_{r,s}$ and
	\begin{align}
	\label{tizio}
	\norm{\es H}_{r,s} & \le 2 \norm{H}_{r+\rho,s}\,,
	\\
	\label{caio}
	\norm{\pa{\es - \id}H}_{r,s}
	&\le  \delta^{-1}
	\norm{S}_{r+\rho,s}
	\norm{H}_{r+\rho,s}\,,
	\\
	\label{sempronio}
	\norm{\pa{\es - \id - \set{S,\cdot}}H}_{r,s} &\le 
	\frac12 \delta^{-2}
	\norm{S}_{r+\rho,s}^2
	\norm{H}_{r+\rho,s}	\end{align}
	More generally for any $h\in\N$ and any sequence  $(c_k)_{k\in\N}$ with $| c_k|\leq 1/k!$, we have 
	\begin{equation}\label{brubeck}
	\norm{\sum_{k\geq h} c_k \ad^k_S\pa{H}}_{r,s} \le 
	2 \|H\|_{r+\rho,s} \big(\|S\|_{r+\rho,s}/2\delta\big)^h
	\,,
	\end{equation}
	where  $\ad_S\pa{\cdot}:= \set{S,\cdot}$.
\end{lemma}
\begin{lemma}\label{estim-cohom}
	Fix $s\geq 0$ and $\s>0$ and $\omega\in \dg$.	For any $R\in\cH_{r,s}^{d}$ with $d\ge 1$ and such that $\Pi_{\cK} R=0$,  the Homological equation $L_\omega S = R$ has a unique solution $S=L_\omega^{-1}R\in\cH_{r,s+\s}^d$ such that $\Pi_{\cK} S=0$ and moreover
	\begin{equation}\label{cavolfiore gevrey}
	\norm{L_\omega ^{-1} R}_{r,s+\s}\le \g^{-1} e^{\Cuno\s^{-\frac{3}{\theta}}}\norm{R}_{r,s}	
	\end{equation}
\end{lemma}

\subsection{Poof of the main Theorem}

The theorem follows by the following holomorphic version of Lemma \ref{Birlinform}.
	If  $H_0\in \cF\cap \cH_{\tr,s_0}$ of the form $D_\omega+ P_0$ with $P_0\in \cF^{\ge 1}$ is formally linearizable.

Fix $0< r_0<\tr$ and $s_0>0$ so that
\[
\e_0:=\g^{-1}	\|P_0\|_{r_0,s_0} \le \g^{-1}\frac{r_0}{\tr}\|P_0\|_{\tr,s_0}
\]
is appropriately small.
More precisely, fix $\tC = 1+ \pi^2/6$ and  assume
\begin{equation}
\label{piccolo1}
\e_0^{-1} \ge  \tK\sup_n  e^{\cC_2(s_0) n^{\frac6\theta}} n^2 \max(e^{n-\chi^n}, e^{-(2-\chi)\chi^n}) .
\end{equation}
where $\tK$ is an appropriately large absolute constant while $\cC_2(s_0)= \cC_1 {\tC}^{\frac{3}{\theta}}{s_0}^{-\frac{3}{\theta}}$.
\\
Let
\[
r_i = r_{i-1}-  \rho_{i-1}\,,\quad s_{i} = s_{i-1} + \s_{i-1} \,,\quad  \td_i = 2^{i}\,, \rho_i =\frac{ r_0}{2\tC \jap{i}^2}\,,\quad \s_i = \frac{ s_0}{\tC \jap{i}^2}
\]
so that $r_i \to r_0/2$ and $s_i\to 2s_0$.

Fix $1<\chi<2$ such that\footnote{for example if $\chi= 15/14$  the sup on the left hand side is smaller than $-0,2$.}
\begin{equation}
\label{chi}
\sup_{n\ge 0}2^{n+1}\ln(1- \frac{1}{2\tC n^2} )+\chi^n(\chi-1) \le -0.1 \,
\end{equation}

\begin{lemma}\label{Birlinholo}

	 The sequence of generating functions and Hamiltonians of Lemma \ref{Birlinform}.
	\[
	\{D_\omega,S_i\} =\Pi^{\cR}  H_{i}\,,\quad H_{i}= e^{\{S_{i-1},\cdot\} }H_{i-1}\,.
	\]
	satisfies
	\[
	H_i = D_\omega + P_i \,,\quad P_i\in \cF^{\ge \td_i} \cap \cH_{r_i,s_i}\,.
	\]
	with the bounds
	\[
\|S_{i-1}\|_{r_{i-1},s_i} \le  \g^{-1}e^{\cC_1 \s_{i-1}^{-\frac3\theta}}\|P_{i-1}\|_{r_{i-1},s_{i-1}} \,,\quad	\| P_i\|_{r_i,s_i} \le \|P_0\|_{r_0,s_0} e^{-\chi^i} \,.
	\]
	Moreover each $S_{i-1}$ defines a symplextic analytic change of variables $\Psi_{i-1}: B_{r_{i}}(\th_s) \to B_{r_{i-1}}(\th_s)$ for all $s\ge s_i$ satisfying
	\begin{equation}
	\label{super}
	\sup_{|u|_s\leq r_i}|\Psi_i(u) -u|_s \le 2^{-i} r_0
	\end{equation}
	Finally setting
	\[
	\Phi_i= \Psi_1\circ\Psi_2\circ\dots \Psi_i
	\]
	we have that
	$\Phi_i\to \Phi_\infty$ where $\Phi_\infty$ is an invertible  symplectic map $B_{r_0/2}(\th_{2s_0}) \to B_{r_0}(\th_{2s_0})$
	such that
	\[
	H_0\circ \Phi_\infty = D_\omega
	\]
\end{lemma}
\begin{proof}
By induction. Let us denote $\g^{-1}\|P_0\|_{r_0,s_0}:= \e_0$. Fix $k\ge 0$ and assume that for all $i\le k$ the Lemma holds. 
By definition
\[
S_k = \ad_{D_\omega}^{-1}\Pi_\cR P_k\,.
\]
\\
For all $s\ge s_k+\s_k\equiv  s_{k+1}$, by Lemma \ref{estim-cohom} and \eqref{piccolo1}
\[
\|S_k\|_{r_k,s} \le \|S_k\|_{r_k,s_{k+1}} \le \g^{-1} e^{\cC_1 \s_{k}^{-\frac3\theta}}\|P_{k}\|_{r_{k},s_{k}}\le \e_0 e^{\cC_2(s_0) k^{\frac6\theta}} e^{-\chi^k}\le \frac{1}{16 e 2\tC k^2} \le \frac{\rho_k}{8 e r_k} 
\]

so, by Lemma \ref{ham flow} the time one flow $\Psi^1_{S_k}: B_{r_{k+1}}(\th_s)\to  B_{r_k}(\th_s)$ is well defined analytic, symplectic 
and, by \eqref{pollon} satisfies
\begin{equation}
\sup_{u\in  B_{r_{k+1}}(\th_{s})} 	\abs{\Phi^1_{S_k}(u)-u}_{\th_{s}}
\le
r_k \norm{S_k}_{r_k,s} \le C \e_0 r_0 k^{-2}e^{\cC_2(s_0) k^{\frac6\theta}} e^{-\chi^k}\stackrel{\eqref{piccolo1}}{\le}  2^{-k}r_0\,.\label{pollonk}
\end{equation}
Recalling that 
\begin{align*}
H_{k+1}:=e^{\{S_k,\cdot\} }H_{k} &=D_\omega + P_k + \{S_k, D_\omega\} + \sum_{h=2}^\infty \frac{\ad_{S_k}^{h-1}}{h!} \{S_k,D_\omega \}
+ \sum_{h=1}^\infty \frac{\ad_S^h}{h!} P_k \\
&= D_\omega+  \Pi^\cK P_k  - \sum_{h=1}^\infty \frac{\ad_{S_k}^{h}}{(h+1)!}\Pi^{\cR}  P_{k}  + \sum_{h=1}^\infty \frac{\ad_{S_k}^h}{h!} P_k\,\\
&=: D_\omega+ \Pi ^{<2^{k+1}} \Pi^\cK P_k + P_{k+1}\,.
\end{align*}
and that in Lemma \ref{Birlinform} we have proved that $\Pi ^{<2^{k+1}}\Pi^\cK P_k=0$, we get
\[
P_{k+1}= \Pi ^{\ge 2^{k+1}}\Pi^\cK P_k  - \sum_{h=1}^\infty \frac{\ad_{S_k}^{h}}{(h+1)!}\Pi^{\cR}  P_{k}  + \sum_{h=1}^\infty \frac{\ad_{S_k}^h}{h!} P_k
\] 
Now
\begin{align*}
\|\Pi ^{\ge 2^{k+1}}\Pi^\cK P_k\|_{r_{k+1},s_{k+1}} \le \pa{\frac{ r_{k+1}}{r_k}}^{\td_{k+1}}\|P_k\|_{r_k,s_k} &\le \e_0(1- \frac{1}{2\tC k^2})^{2^{k+1}} e^{-\chi^k}
\\
\|\sum_{h=1}^\infty \frac{\ad_{S_k}^{h}}{(h+1)!}\Pi^{\cR}  P_{k}  + \sum_{h=1}^\infty \frac{\ad_{S_k}^h}{h!} P_k\|_{r_{k+1},s_{k+1}}
&\le \frac{16 e r_k}{ \rho_k} \|P_k\|_{r_k,s_k} \|S_k\|_{r_{k+1},s_{k+1}} \\ & \le C \e^2_0 e^{\cC_2(s_0) k^{\frac6\theta}} e^{-2\chi^k} k^2
\end{align*}
The bound on $P_{k+1}$ follows from \eqref{piccolo1} and \eqref{chi} which imply 
\[
\e_0(1- \frac{1}{2\tC k^2})^{2^{k+1}} e^{-\chi^k} + C \e^2_0 e^{\cC_2(s_0) k^{\frac6\theta}} e^{-2\chi^k} k^2 \le \e_0 e^{-\chi^{k+1}}
\]
In order to prove the convergence we remark that all the $\Psi_i$ map $ B_{r_{i}}(\th_{2s_0}) \to B_{r_{i-1}}(\th_{2s_0})$, consequently $\Phi_i$ maps 
$ B_{r_{i}}(\th_{2s_0}) \to B_{r_0}(\th_{2s_0})$ and, by \eqref{super}, it is a Cauchy sequence.
\end{proof}

\appendix
\section{Technical Lemmata}
In the following, we adapt material from \cite{BMP18} to non mass conservation situation.
\subsection{Proof of Lemmata \ref{gasteropode} and \ref{genoveffa}}
We follow here \cite{BMP18}[Appendix B. Proof of lemma 3.1]. For any $H\in \cH_{r,s}$ (we recall that this space depends on two extra parameters $p\geq \frac{1}{2}$ and $0<\theta\leq 1$)  we define  a map
\[
B_1(\ell^2)\to \ell^2 \,,\quad y=\pa{y_j}_{j\in \Z }\mapsto 
\pa{Y^{(j)}_{H}(y;r,s)}_{j\in \Z}
\]
by setting
\begin{equation}\label{giggina}
Y^{(j)}_{H}(y;r,s) := \sum_\ast |H_{\bal,\bbt}| \frac{(\bal_j+\bbt_j)}{2}c^{(j)}_{r,s}(\bal,\bbt) y^{\bal+\bbt-e_j}
\end{equation}
where  $e_j$ is the $j$-th basis vector in $\N^\Z$, while the coefficient
\begin{equation}
\label{persico}
c^{(j)}_{r,s}(\al,\bt)=  r^{|\al|+|\bt|-2} \pa{\frac{\jap{j}^2}{\prod_i\jap{i}^{\bal_i+\bbt_i}}}^{p} e^{-s (\sum_i \jap{i}^\theta (\bal_i+\bbt_i) -2\jap{j}^\theta)}
\end{equation}

For brevity, we set
$$
\sum_\ast:=\sum_{\bal,\bbt: \;\pi(\al,\bt)=0}\,.
$$

The vector field $Y_H$ is a majorant  analytic function on $\ell^2$ which has the {\it same norm as $H$}. Since the majorant  analytic functions on a given space have a natural ordering this gives us a natural criterion for immersions, as formalized in the following Lemma.
\begin{lemma}\label{stantuffo}	 Let
$\ri,\rs>0,\,s,s'\geq 0.$ The following properties hold.
\begin{enumerate}
	\item   The norm of $H$ can be expressed as
	\begin{equation}\label{ypsilon}
	\norm{H}_{r,s}= 
	\sup_{|y|_{\ell^2}\le 1}\abs{Y_H(y;r,s)}_{\ell^2}
	\end{equation}
	\item  Given 
	$
	H^{(1)}\in \cH_{\rs,s'}$ 
	and $H^{(2)}\in \cH_{\ri,s}\,,
	$
	\\	
	such that for all $\bal,\bbt\in \N^\Z_f$ and  $j\in \Z$ with $\bal_j+\bbt_j\neq 0$ 
	one has
	\[
	|H^{(1)}_{\bal,\bbt}| c^{(j)}_{\rs,s'}(\bal,\bbt)  
	\le 
	c
	|H^{(2)}_{\bal,\bbt}| c^{(j)}_{\ri,s}(\bal,\bbt),
	\]
	for some $c>0,$
	then
	\[
	\norm{H^{(1)}}_{\rs,s'}
	\le 
	c
	\norm{H^{(2)}}_{\ri,s}\,.
	\]
\end{enumerate}
\end{lemma}
\begin{proof}[Proof of Lemma \ref{gasteropode}]
	Recalling \eqref{persico}, we have
	\[
	\frac{c^{(j)}_{\rs,s}(\bal,\bbt)}{c^{(j)}_{r,s}(\bal,\bbt)}= \pa{\frac{\rs}{r}}^{|\bal|+|\bbt|-2}\,.
	\]
	Since $|\bal|+|\bbt|-2\ge \td$, the inequality
	follows by Lemma \ref{stantuffo} with $H^{(1)}=H^{(2)}$ and $s=s'$.
\end{proof}
In order to prove Lemma \ref{genoveffa} we need some notations and results proven in \cite{Bourgain:2005} and \cite{Yuan_et_al:2017}.
\begin{defn}\label{n star}
	Given a vector $v=\pa{v_i}_{i\in \Z}\in \N^\Z_f$ with $|v|\ge 2$  we denote by $\na=\na(v)$ the vector $\pa{\na_l}_{l\in I}$ (where $I\subset \N$ is finite)  which is the decreasing rearrangement 
	of
	$$
	\{\N\ni h> 1\;\; \mbox{ repeated}\; v_h + v_{-h}\; \mbox{times} \} \cup \set{ 1\;\; \mbox{ repeated}\; v_1 + v_{-1} + v_0\; \mbox{times}  }
	$$
\end{defn}
\begin{rmk}
	A good way of envisioning this list is as follows. Given  an infinite set of variables $\pa{x_i}_{i\in\Z}$ and a vector $v=\pa{v_i}_{i\in \Z}\in \N^\Z_f$ consider the monomial $x^v:= \prod_i x_i^{v_i}$. We can write 
	\[
	x^v= \prod_i x_i^{v_i} = x_{j_1} x_{j_2}\cdots x_{j_{|v|}}\,,\quad \mbox{ with}\quad j_k\in \Z
	\] 
	then $\na(v)$ is the decreasing rearrangement of the list $\pa{\jap{j_1},\dots,\jap{j_{|v|}}}$.
	%
\end{rmk}
\begin{ex}
Let us set 
		$$
		v_{-1}=2, v_{0}=3, v_1= 1, v_{3}=1, v_{4}= 2.
		$$
		Hence, $1$ is repeated 6 times, $3$ is repeated 1 time, and $4$ is repeated 2 times~:
		$$
		\hat n_1=4,\hat n_2= 4, \hat n_3 = 3, \hat n_4=\dots=\hat n_9=1
		$$
\end{ex}

Given $\bal,\bbt\in\N^\Z_f$ 
with $ |\bal|+|\bbt|\ge 2$
from now on we define
$$
\na=\na(\bal+\bbt)\,
\qquad \mbox{and set}\quad
N:=|\bal|+|\bbt| 
$$
which is the cardinality of $\na.$ 
We  observe that, $N\ge 2$ and since
\begin{equation}
\label{moment}
0= \sum_{i\in \Z} i\pa{\bal_i - \bbt_i}= \sum_{h> 0} h \pa{\bal_h - \bbt_h - \bal_{-h} + \bbt_{-h}} \,,
\end{equation}
there exists a choice of $\s_i = \pm1, 0$ such that\footnote{A given $h>1$ appears  ${\bal_h + \bbt_h + \bal_{-h} + \bbt_{-h}}$ times in the list $\na$. Thus in order to get the summand $ h \pa{\bal_h - \bbt_h - \bal_{-h} + \bbt_{-h}}$ we assign to the $\na_l$ with $\na_l=h$ the sign $\s_l=+ $, $\al_h+\bt_{-h}$ times and the sign  $\s_l=- $, $\al_{-h}+\bt_{h}$ times. Let us now consider the case $h=1$. By construction, $1$ appears ${\baluno + \bbtuno + \bal_{-1} + \bbt_{-1}+ \al_0 +\bt_0}$ times in $\na$. Thus in order to obtain the summand $ \pa{\baluno - \bbtuno - \bal_{-1} + \bbt_{-1}}$ we assign to the $\na_l$ with $\na_l=1$ the sign $\s_l=+ $, $\al_1+\bt_{-1}$ times, the sign  $\s_l=- $, $\al_{-1}+\bt_{1}$ times and $\s_l=0$ the remaining $\al_0 +\bt_0$ times. }
\begin{equation}\label{pi e cappucci}
\sum_l \sigma_l\na_l=0.
\end{equation}
with $\sigma_l \neq 0$  if $\na_l \neq 1$.
Hence, 
\begin{equation}\label{eleganza}
\na_1\le\sum_{l\ge 2}\na_l.
\end{equation}
Indeed, if $\sigma_1 = \pm 1$, the inequality follows directly from \eqref{pi e cappucci}; if $\sigma_1 = 0$, then $\na_1=1$ and consequently $\na_l = 1\, \forall l$. Since
$|\al|+|\bt|\ge 2$, the list $\na$ has at least two elements, so the inequality is achieved.
\begin{lemma}\label{constance generalbis}
	Given $\bal,\bbt$ such that $\sum_i i (\bal_i-\bbt_i)=0$, and $|\al|+|\bt|\ge 2$, we have that setting $\na=\na(\bal+\bbt)$
	\begin{equation}\label{yuan 2bis}
	\sum_i \jap{i}^\theta(\bal_i+\bbt_i) =\sum_{l\ge 1} \na_l^\theta  \ge 2 \na^\theta_1+ (2-2^\teta) {\sum_{l\ge 3} \na_l^\theta} .
	\end{equation}
\end{lemma}
\begin{proof}
The lemma above was proved in  \cite{Bourgain:2005} for $\theta=\frac12$  and for general {$0<\theta<1$} in \cite{Yuan_et_al:2017}[Lemma 2.1], in the case of zero mass and momentum. For completeness we give below a  proof , using only momentum conservation.
\\
We start by noticing that if $|\al|+|\bt|= 2$ then  $\na$ has cardinality equal to two and \eqref{yuan 2bis} becomes $\na_1+\na_2 \ge 2\na_1$.  Now,  by \eqref{eleganza}, momentum conservation implies that
$\na_1=\na_2$ and hence \eqref{yuan 2bis}.
\\
 If  $|\al|+|\bt|\ge 3$ we write
\[
\sum_i \jap{i}^\theta(\bal_i+\bbt_i) -2\na_1^\theta=\sum_{l\ge 2} \na_l^\theta  - \na_1^\theta \ge  \sum_{l\ge 2} \na_l^\theta  - (\sum_{l\ge 2} \na_l)^\theta
\]
since the cardinality of $\na$ is at least three we may write
\[
\sum_{l\ge 2} \na_l^\theta  - (\sum_{l\ge 2} \na_l)^\theta = \na_2^\theta + \sum_{l\ge 3} \na_l^\theta - (\na_2+\sum_{l\ge 3} \na_l)^\theta 
\]
Now setting, {for $x_i\geq 1$, $i=2,\ldots, N$,}
\[
f(x_2,\dots,x_N):= x_2^\theta +(2^\theta - 1) \sum_{l\ge 3} x_l^\theta - (x_2+\sum_{l\ge 3} x_l)^\theta. 
\]
{Hence, we have $\partial_{x_2} f\ge 0$  for  $x_2\ge x_3\ge 1$.} 
Then 
$$
f(x_2,\dots,x_N)\ge  f(x_3,x_3,x_4,\dots,x_N)=: f_3(x_3,\dots,x_N)\,.
$$ 
Now we set \[f_n(x_n,\dots ,x_N):= f(\underbrace{x_n,\dots,x_n}_{n-1},x_{n+1},\dots,x_N)= (1+ (2^\theta -1)(n-2))x_n^\theta  +\sum_{\ell\ge  n+1}x_\ell- ((n-1)x_n+ \sum_{\ell\ge  n+1}x_\ell)^\theta
\]
so that $f(x_2,\dots,x_N)\ge  f_3(x_3,\dots,x_N)$. Assume inductively that for some $3\le n<N$, one has $f(x_2,\dots,x_N)\ge f_3(x_3,\dots,x_N) \ge \dots \ge f_{n}(x_n,\dots ,x_N)$. By direct computation \footnote{recalling  that  the $x_\ell>0$ and that  $ 1+(2^\theta-1)k - (k+1)^\theta \ge 0$, with $k= n+2>1$}
\begin{align*}
\partial_{x_n} f_n&=\theta\Big[ \frac{(1+ (2^\theta -1)(n-2))}{x_n^{1-\theta }} -\frac{n-1}{((n-1)x_n+ \sum_{\ell\ge  n+1}x_\ell)^{1-\theta} }\Big]\\
&\ge \theta x_n^{\theta-1 }\Big[ {(1+ (2^\theta -1)(n+2))} -{(n-1)^{\theta} }\Big]\ge 0\,,
\end{align*}
		so that the minimum is attained in $x_n= x_{n+1}$ and $f(x_2,\dots,x_N) \ge  f_{n+1}(x_{n+1},\dots ,x_N)$. In conclusion
\[
f(x_2,\dots,x_N)\ge f(x_N,\dots,x_N)\ge 0
\]
where the last inequality follows by recalling  that $ 1+(2^\theta-1)k - (k+1)^\theta \ge 0$
 for $k\geq 1$. 
\end{proof}
The Lemma proved above, is fundamental in discussing the properties of $\cH_{r}(\th_{p,s,a})$ with $s>0$, indeed it implies
\begin{equation}\label{stima1}
\sum_i \jap{i}^\theta (\bal_i+\bbt_i) -2\jap{j}^\theta\ge(2-2^\theta) \pa{\sum_{l\ge 3} \na_l^\theta } 
\ge  0
\end{equation}
for all $\bal,\bbt$ such that $\bal_j+\bbt_j\ne 0$. Indeed, this follows from the fact that $\jap{j} \leq \na_1$. 

\begin{proof}[Proof of Lemma \ref{genoveffa}]
	In all that follows we shall use systematically the fact that our Hamiltonians {are momentum preserving},  are zero at the origin and have no linear term so that $\abs{\bal} +\abs{\bbt} \ge 2$.
	\\
  We need to show that
	\begin{equation}\label{gigina}
	\frac{	c^{(j)}_{r,s+\s}(\bal,\bbt)}{c^{(j)}_{r,s }(\bal,\bbt)} = 	\exp(-\s (\sum_i \jap{i}^\theta (\bal_i+\bbt_i) -2\jap{j}^\theta) \le 1\,.
	\end{equation}
	The first identity comes form \eqref{persico}, while the last inequality  follows by \eqref{stima1} of Lemma \ref{constance generalbis} 
	\end{proof}
\subsection{Proof of Lemma \ref{ham flow}}
We recall the following classical result.
\begin{lemma}\label{palis}
	Let $0<r_1<r.$ Let $E$ be a Banach space endowed with the norm $|\cdot|_E$.
	Let $X:B_r \to E$ a vector field satisfying
	$$\sup_{B_r}|X|_E\leq \delta_0\,.$$
	Then the flow $\Phi(u,t)$ of the vector field\footnote{Namely the solution 
		of the equation $\partial_t \Phi(u,t)=X(\Phi(u,t))$ with initial datum
		$\Phi(u,0)=u.$} is well defined for every 
	$$|t|\leq T:=\frac{r-r_1}{\delta_0}$$
	and $u\in B_{r_1}$
	with estimate
	$$
	|\Phi(u,t)-u|_E\leq \delta_0 |t|\,,\qquad
	\forall\, |t|\leq T
	\,.
	$$
\end{lemma}

\begin{proof}[Proof of Lemma \ref{ham flow}]
	The estimate for the Poisson bracket is proven in \cite{BBiP1}. In order to prove the other estimates we use Lemma \ref{palis}, with $E\to \th_{s}$, $X\to X_S$,
	$\delta_0\to (r+\rho) |S|_{r+\rho},$ $r\to r+\rho,$ $r_1\to r,$	
	$T\to 8e.$  finally we do not write the dependence on $s$ which is fixed.
	
	Then the fact that the time $1$-Hamiltonian flow 
	$\Phi^1_S: B_r(\th_{s})
	\to B_{r + \rho}(\th_{s})$  is well defined, analytic, symplectic  follows,	 since 
	\[
	\sup_{u\in  B_{r+\rho}(\th_{s})}
	|X_S|_{\th_{s}}
	\le (r+\rho) |S|_{r+\rho}<\frac{\rho}{8 e}\,.
	\]
	Regarding the estimate \eqref{pollon}, again by 
	Lemma \ref{palis}  (choosing $t=1$), we get
	\[
	\sup_{u\in  B_{r}(\th_{s})}
	\abs{\Phi^1_S(u)-u} _{\th_{s}}
	\le
	(r+\rho) |S|_{r+\rho}
	<\frac{\rho}{8 e}
	\,.
	\]

	Estimates \eqref{tizio},\eqref{caio},\eqref{sempronio} directly follow by \eqref{brubeck} with $h=0,1,2,$
	respectively and $c_k=1/k!$, 
	recalling that by Lie series 
	$$
	H \circ \Phi^1_S = e^{\rm ad_S} H = \sum_{k=0}^\infty \frac { {\rm ad}_S^k H}{k!} =
	\sum_{k=0}^\infty \frac {  H^{(k)}}{k!}\,,
	$$
	where
	$ H^{(i)} := {\rm ad}_S^i (H)= {\rm ad}_S ( H^{(i-1)}) $,  $ H^{(0)}:=H $.\\
	Let us prove \eqref{brubeck}.
	Fix $k\in\N,$ $k>0$ and set
	$$
	r_i := r +\rho(1 - \frac{i}{k}) \,
	\, ,  \qquad  i = 0,\ldots,k \, .
	$$
	Note that, by the immersion properties of the norm in Lemma \ref{gasteropode},
	\begin{equation}\label{morello}
	\|S\|_{r_i}\leq \|S\|_{r+\rho}\,,\qquad
	\forall\,  i = 0,\ldots,k\,.
	\end{equation}
	Noting that
	\begin{equation}\label{dave}
	1+\frac{k r_i}{\rho} \,
	\leq
	k \pa{1+\frac{r}{\rho }}\,,
	\qquad \forall\, i=0,\ldots,k\,,
	\end{equation}
	by using $k$ times \eqref{commXHK}  we have
	\begin{eqnarray*}
		\| {H^{(k)}}\|_r
		&=& 
		\|   \{S, {H^{(k-1)}}\} \|_r
		\leq  
		4 (1+\frac{ k r}{\rho})
		\|{H^{(k-1)}}\|_{r_{k-1}}\|
		S\|_{r_{k-1}}
		\\
		&\stackrel{\eqref{morello}}\leq&
		\|H\|_{r+\rho}
		\|S\|_{r+\rho}^k
		4^k
		\prod_{i=1}^k
		(1+\frac{ k r_i}{\rho})
		\stackrel{\eqref{dave}}\leq
		\|H\|_{r+\rho}
		\left(
		4k \pa{1+\frac{r}{\rho }}\|S\|_{r+\rho}
		\right)^k
		\,.
	\end{eqnarray*} 
	Then, using $ k^k\leq e^k k!, $
	we  get
	\begin{eqnarray*}
		\left\|\sum_{k\geq h} c_k {H^{(k)}}\right\|_{r} &\leq&
		\sum_{k\geq h} |c_k| \|{H^{(k)}}\|_{r}
		\leq
		\|H\|_{r+\rho} \sum_{k\geq h} \left(
		4e \pa{1+\frac{r}{\rho }}\|S\|_{r+\rho}
		\right)^k
		\\
		&= & \| H\|_{r+\rho} \sum_{k\geq h} (\|S\|_{r+\rho}/2\delta)^k
		\stackrel{\eqref{stima generatrice}}\leq 2 \|H\|_{r+\rho} (\|S\|_{r+\rho}/2\delta)^h\,.
	\end{eqnarray*}	
	Finally, if $ S $ and $ H $ satisfy momentum conservation so does  each
	$ {\rm ad}_S^k H $, $ k \geq 1 $, hence $ H \circ \Phi^1_S $ too.
\end{proof}
\subsection{Proof of lemma \ref{estim-cohom}}
 Here we strongly use the fact that we are working with a dispersive PDE on the circle with superlinear dispersion law.
 
 	By Lemma \ref{stantuffo} (2), 
 we have
 \[
 \norm{L_\omega^{-1} R}_{r,s+\s}\le \g^{-1}K \norm{ R}_{r,s}
 \] where
 \begin{equation*}
 	K=\g\sup_{\substack{j: \bal_j+\bbt_j\neq 0\\  \pi(\al,\bt)=0}}
 	\frac{e^{-\s\pa{\sum_i\jap{i}^\theta (\bal_i+\bbt_i) -2\jap{j}^\theta}}}{\abs{\omega\cdot{\pa{\bal - \bbt}}}}
 	\,.
 \end{equation*}
 Therefore proving \eqref{cavolfiore gevrey} amounts  to showing that
 \begin{equation}
 	\label{pinne}
 	K \le e^{\cC_1 \s^{-\frac 3\theta}}\,.
 \end{equation}
We divide in  two cases  regarding whether the inequality 
\begin{equation}\label{divisor}
	\abs{\sum_i{\pa{\bal_i-\bbt_i}i^2}}\le 2 \sum_i\abs{\bal_i-\bbt_i}\,,
\end{equation}
holds or not.
We remark that 
\begin{equation}\label{fiandre}
	\abs{\sum_i{\pa{\bal_i-\bbt_i}i^2}}\ge 2 \sum_i\abs{\bal_i-\bbt_i}
	\qquad
	\Longrightarrow
	\qquad
	\abs{\omega\cdot\pa{\bal-\bbt}}\ge 1\,,
\end{equation}
indeed denoting $\omega_j = j^2 + \xi_j $ with $\abs{\xi_j}\le \frac{1}{2}$, \[\abs{\omega\cdot\pa{\bal-\bbt}} \ge 2\sum_j\abs{\bal_j - \bbt_j} - \frac{1}{2}\sum_j\abs{\bal_j - \bbt_j}\ge 1.\]
Of course if 
$\abs{\omega\cdot\pa{\bal-\bbt}}\ge 1$, by 
\eqref{stima1} and \eqref{gigina}
we get 
$$
\g
\frac{e^{-\s\pa{\sum_i\jap{i}^\theta (\bal_i+\bbt_i) -2\jap{j}^\theta}}}{\abs{\omega\cdot{\pa{\bal - \bbt}}}}
\leq 1\,
$$
and the bound \eqref{pinne} is trivially achieved.\\
Otherwise, to  deal with  the case in which \eqref{divisor} holds, we need some notation.
Given $u\in \Z^\Z_f$,  consider the set
\[
M(u):=\set{j\neq 0 \,,\quad \mbox{repeated}\quad  \abs{u_j} \;\mbox{times}}\,,
\]
where $D(u)<\infty$ is its cardinality.
Define the vector $m=m(u)$ as the reordering of the elements of the set above 
such that
$|m_1|\ge |m_2|\ge \dots\geq |m_D|\ge 1.$ 	
\\
Given $\bal\neq\bbt\in\N^\Z_f$ with $|\bal|+|\bbt|\ge 3$
we consider $m=m(\bal-\bbt)$ and $\na=\na(\bal+\bbt).$	
If we denote by $D$ the cardinality of $m$ and $N$ the one of $\na$ we have 
\begin{equation}\label{cappella}
D+\bal_0+\bbt_0\le N
\end{equation}
and
\begin{equation}\label{abbacchio}
(|m_1|,\dots,|m_D|,\underbrace{1,\;\dots \;,1}_{N-D\;\rm{times}} )\, \preceq\,
\pa{\na_1,\dots \na_N}\,.
\end{equation}
\begin{ex}
	Let set $v=\al+\beta$ and $u=\al-\beta$ with
		\begin{eqnarray*}
			\al_{-5}=1, \al_{-2}=2, \alpha_0= 2,\alpha_4=1 &&\\
			\beta_{-5}=1, \beta_{-3}=2, \beta_0=3,\beta_6=1 &&\\
			\pi(\al,\beta)=(-5)(1-1)+(-3)(-2)+(-2)(2)+4(1)+6(-1)=0&&\\
			v_{-5}=2,v_{-3}=2, v_{-2}=2, v_{0}=5, v_{4}=1, v_{6}=1&&\\
			u_{-5}=0,u_{-3}=-2, u_{-2}=2, u_{0}=-1, u_{4}=1, u_{6}=-1&&\\
			\na (v)=(6, 5,5,4,3,3,2,2,1,1,1,1,1), N= 13(=\text{Card}(\hat n))&&\\
			M(u)=\{-3, -3, -2, -2,4,6\}, m(u)=\{6,4, -3, -3, -2, -2\}, D(u)=6. 
		\end{eqnarray*}
		Therefore, we have $D(u)+\alpha_0+\beta_0=8\leq 13=N(\hat n(v))$. Hence, (\ref{cappella}) holds. 
		\\
		Futhermore,
		$(6,4, 3, 3, 2, 2, 1,1,1,1,1,1,1)\leq \hat n(v)$, that is  (\ref{abbacchio}).
\end{ex}	

\begin{lemma}\label{mizza}
	Assume that $g$ defined on $\Z$ is non negative,  even
	and not decreasing on $\N.$ 
	Then, if $\bal\neq\bbt$,
	\begin{equation}\label{pula}
	\sum_{i\in\Z} g(i) |\bal_i-\bbt_i|
	\leq
	2g(m_1)+
	\sum_{l\geq 3} g(\na_l)\,.
	\end{equation}
\end{lemma}
\begin{proof}
	By definition of $m(\al-\bt)$ and setting 
	$
	\s_l= {\rm sign}(\bal_{m_l}-\bbt_{m_l})\,,
	$
	we have
	\begin{eqnarray}\label{pula2}
	\sum_{i\in\Z} g(i) (\bal_i-\bbt_i)
	&=&
	g(0)(\bal_0-\bbt_0)+
	\sum_{l\geq 1} \s_l g(m_l)\,.
	\end{eqnarray}
	Hence
	\begin{eqnarray*}
		\sum_{i\in\Z} g(i) |\bal_i-\bbt_i|
		&=&
		g(0)|\bal_0-\bbt_0|+
		\sum_{l\geq 1} g(m_l)
		\\
		&\leq& 
		g(1)(\bal_0+\bbt_0)+
		2g(m_1)+
		\sum_{l\geq 3} g(m_l)
	\end{eqnarray*}
	and \eqref{pula} follows by
	\eqref{cappella} and \eqref{abbacchio}.
\end{proof}

By  \eqref{pula2}
\begin{equation}\label{somma sigma p}
\mic{0=} \sum_{i\in \Z}\pa{\bal_i-\bbt_i}i =  \sum_l \s_l m_l 
\end{equation}
and
\begin{equation} \label{somma sigma quadro}
{\sum_i{\pa{\bal_i-\bbt_i}i^2}}=
\sum_l \s_l m^2_l\,.
\end{equation}
Analogously 
\begin{equation}
\label{enne}
{\sum_i{\abs{\bal_i-\bbt_i}}}
= D+|\bal_0-\bbt_0|
\stackrel{\eqref{cappella}}\leq N\,.
\end{equation}
Finally note that 
\begin{equation}\label{gina}
\sigma_l\sigma_{l'} =-1\qquad \Longrightarrow\qquad m_l \neq m_{l'}\,.
\end{equation}

\begin{lemma}\label{orbo}
	Given $\bal\neq\bbt\in\N^\Z_f,$ such that $\pi(\bal-\bbt)=0$, 
 $N\ge 3,D\ge 1 $	and satisfying  \eqref{divisor}, we 
	have
	\begin{equation}
	\label{chiappechiare}
	\abs{m_1}\le  7\sum_{l\ge 3}\na_l^2\,. 
	\end{equation}
\end{lemma}

\begin{proof} 
The case $D=1$ is not compatible with momentum conservation.
	Let us now consider the case $D=2$,
	i.e.
	$$
	\bal-\bbt=\s_1 \be_{m_1}+\s_2 \be_{m_2}
	+(\bal_0-\bbt_0)\be_0\,.
	$$ 
If 
	$\s_1\s_2=-1$, momentum conservation imposes $m_1=m_2$  but this contradicts \eqref{gina}.
In the case 
	$\s_1\s_2=1,$ by momentum conservation we have $m_1= -m_2$.
	Then  conditions \eqref{divisor} and \eqref{enne} imply that 
	\[ 
	m_1^2 + m_2^2 \le2 (D +|\bal_0-\bbt_0|) \stackrel{\eqref{enne}}{\le } 2 N \le 6(N-2) \le 6 \sum_{l=3}^N \na_l^2
	\]
since $\na_l\ge 1$.

	Let us now consider the case $D \ge 3$. 
	By \eqref{divisor},\eqref{somma sigma quadro}
	and \eqref{enne}
	\begin{eqnarray*}
		m_1^2 +\s_1\s_2 m_2^2 
		&\le& 
		2(D+|\bal_0-\bbt_0|) + \sum_{l=3}^Dm_l^2
		\le
		2 N + \sum_{l=3}^Dm_l^2 
\le 
		2 N + \sum_{l=3}^N\na_l^2 
		{\le} 7\sum_{l=3}^N
		\na_l^2\,.
	\end{eqnarray*}
since (recall $N\ge 3$)
$ 2 N\le  6(N-2) \le  6\sum_{l=3}^N
\na_l^2$.

	If $\sigma_1\sigma_2 = 1$ then
	\[ 
	\abs{m_1}, \abs{m_2} \le \sqrt{7\sum_{l\ge 3} \na_l^2}.
	\]
	If  $\s_1\s_2 = -1$
	\[
	(|m_1|+|m_2|)(|m_1|-|m_2|)=
	m_1^2 - m_2^2 \le 
	7\sum_{l\ge 3} \na_l^2.
	\]
	Now, if $\abs{m_1}\ne \abs{m_2}$ then 
	\[
	\abs{m_1} + \abs{m_2} 
	\le 	7\sum_{l\ge 3} \na_l^2.
	\]
	Conversely, if $\abs{m_1} = \abs{m_2}$, by \eqref{gina}, $m_1\ne m_2$, hence $m_1 = - m_2$. By substituting this relation into \eqref{somma sigma p}, we have 
	\[
	2\abs{m_1} \le\sum_{l\ge 3}\abs{m_l} \le  \sum_{l\ge 3}\na_l^2\,,
	\]
	concluding the proof.
\end{proof}

Now the key to proving Lemma \ref{constance 2 gen} is the following.
\begin{lemma}\label{constance 2 gen}
	Consider $\bal,\bbt\in\cM$ with
	$\al\ne \bt$ and $ |\al|+|\bt| \ge 3$.
	If
	\eqref{divisor} holds 
	then
	for all $j$ such that $\bal_j+\bbt_j\neq 0$ one has
	\begin{equation}\label{adele}
		\sum_i\abs{\bal_i-\bbt_i}\jap{i}^{\theta/2} 
		\le 
		C_*
		\pa{\sum_i \pa{\bal_i+\bbt_i}\jap{i}^\teta- 2\jap{j}^\teta  }\,, \qquad C_*= \frac{7}{2-2^\theta}
	\end{equation} 
\end{lemma}
\begin{proof}
Let us first consider the case $D=0$, this means that $\al-\bt= (\al_0-\bt_0)\be_0$ and the left hand side of \eqref{adele} reads
$ |\al_0-\bt_0|$.  By \eqref{stima1} and $N\ge 3$ the right hand side of \eqref{adele} is at least $2-2^\theta$, 
so if  $|\al_0-\bt_0| \le 7 $ the result is trivial. Otherwise we have two cases, if $j=0$ 
\[
|\al_0-\bt_0| \le  2(|\al_0-\bt_0| -2\jap{j}^\theta) \le 2 \pa{\sum_i \pa{\bal_i+\bbt_i}\jap{i}^\teta- 2\jap{j}^\teta  }\,,
\]
Otherwise we remark that if $j\ne 0$, $\al_j+\bt_j\ne 0$ and $\al_j-\bt_j=0$, then $\al_j+\bt_j \ge 2$, then
\[
|\al_0-\bt_0| \le (\al_0+\bt_0) + (\al_j+\bt_j -2) \jap{j}^\theta \le \sum_i \pa{\bal_i+\bbt_i}\jap{i}^\teta- 2\jap{j}^\teta \,.
\]
Now we  consider indices $\al,\bt$ such that $N\ge 3,D\ge 1 $. Here we apply Lemma \ref{orbo}
	Given $\bal,\bbt\in\N^\Z_f,$ as above
	we consider $m=m(\bal-\bbt)$ and $\na=\na(\bal+\bbt).$	

	We have\footnote{Using that for $x,y\geq 0$ and $0\leq c\leq 1$
		we get $(x+y)^c\leq x^c+y^c.$} 
	\begin{eqnarray}\label{cosette}
	\sum_i\abs{\bal_i-\bbt_i}\jap{i}^{\theta/2} 
	&\stackrel{\eqref{pula}}\leq&
	2\abs{m_1}^{\frac{\teta}{2}} +  
	\sum_{l\ge 3} \na_l^{\frac{\teta}{2}}  
	\nonumber
	\\
	& \stackrel{\eqref{chiappechiare}}\le &
	2\pa{7 \sum_{l\ge 3} \na_l^2 }^{\frac{\teta}{2}} +
	\sum_{l\ge 3} \na_l^{\frac{\teta}{2}}   
	\nonumber
	\\
	& \le&
	+ 2(7)^{\frac{\teta}{2}}\sum_{l\ge 3}\na_l^\teta + \sum_{l\ge 3} \na_l^{\frac{\teta}{2}} 
	\nonumber
	\\
	& \le  & 
	\frac{2\sqrt{7}+1}{2-2^\theta}\pa{  (2-2^\teta)\pa{\sum_{l\ge 3}\na_l^\teta }}\,,
	\end{eqnarray}
	Then by Lemma \ref{constance generalbis}  and \eqref{cosette} we get
	\begin{eqnarray*}
		\sum_i\abs{\bal_i-\bbt_i}\jap{i}^{\theta/2}
		&\le &
		\frac{7 }{2-2^\teta}\pa{  \sum_i \jap{i}^\teta\pa{\bal_i +\bbt_i}  - 2\na_1^\teta}\\
		& \le &
		\frac{7 }{2-2^\teta} \sq{\sum_i \jap{i}^\teta\pa{\bal_i +\bbt_i} - 2\jap{j}^\teta}\,,
	\end{eqnarray*}
	proving \eqref{adele}.
\end{proof}
\begin{proof}[Conclusion of the proof of Lemma \ref{estim-cohom}]
 By applying Lemma \ref{constance 2 gen}, since $\omega\in \dg$ we get:
\begin{eqnarray}
&&\g\frac{e^{-\s\pa{\sum_i\jap{i}^\theta (\bal_i+\bbt_i) -2\jap{j}^\theta}}}{\abs{\omega\cdot{\pa{\bal - \bbt}}}}  
\nonumber \stackrel{\eqref{diofantinoBIS}}{\le} e^{-\s\pa{\sum_i\jap{i}^\theta (\bal_i+\bbt_i) -2\jap{j}^\theta}} \prod_i\pa{1+(\bal_i-\bbt_i)^{2}\jap{i}^{{2}}}
\\
&&
\stackrel{\eqref{adele}}{\le}
e^{-\frac{\s}{C_* }\sum_i\abs{\bal_i-\bbt_i}\jap{i}^{\frac{\teta}{2}}}\prod_i\pa{1+(\bal_i-\bbt_i)^{2}\jap{i}^{{2}}}
\nonumber
\\
&& \le  
\exp{\sum_i\sq{-\frac{\s}{C_*} \abs{\bal_i - \bbt_i}\jap{i}^{\frac{\teta}{2}} + \ln{\pa{1 + \pa{\bal_i - \bbt_i}^{2}\jap{i}^{{2}}}}}} 
\nonumber
\\
&& 
=  \exp{\sum_i f_i(\abs{\bal_i-\bbt_i})} 
\label{spa}
\end{eqnarray}
where, for $0<\s\leq 1$, $i\in \Z$ and $x\geq 0$, we defined
$$
f_i(x) := -\frac{\s}{C_*} x\jap{i}^{\frac{\teta}{2}} + \ln{\pa{1 + x^{2}\jap{i}^{{2}}}}\,.
$$
Finally, we have 
\begin{lemma}[\cite{BMP18}Lemma 7.2]\label{pajata}
	Setting
	$$
	i_\sharp:=
	\left(
	\frac{24C_*}{\s\theta}
	\ln
	\frac{12C_*}{\s\theta}
	\right)^{\frac2\theta}\,,
	$$
	we get 
	\begin{equation}\label{gricia}
	\sum_i f_i(|\ell_i|)\leq 
	18
	i_\sharp \ln i_\sharp
		\end{equation}
	for every $\ell\in\Z^\Z_f$.
\end{lemma}
\begin{proof}
		First of all we note that
	$$
	\sum_i f_i(|\ell_i|)=
	\sum_{i\ \text{s.t.} \ \ell_i\neq 0} f_i(|\ell_i|)
	$$
	since $f_i(0)=0.$
	We have that\footnote{Using that
		$\ln(1+y)\leq 1+\ln y$ for every $y\geq 1.$}
	$$
	f_i(x) \leq
	-\frac{\s}{C_* }\jap{i}^{\frac\theta 2}x + {2} \ln(x)+  2\ln \jap{i} +1\,,\qquad
	\forall\, x\geq 1\,.
	$$
Now,
	\begin{equation*}
		\max_{x\geq 1} \left( -\frac{\s}{C_* }\jap{i}^{\frac\theta 2}x + 2 \ln(x)\right)=
		\left\{
		\begin{array}{l} 
			\displaystyle -\frac{\s}{C_* }\jap{i}^{\frac\theta 2} 
			\qquad\qquad\qquad\ \,\qquad\text{if}\quad \jap{i}\geq i_0\,,
			\\
			\empty
			\\  
			\displaystyle
			-{2}+{2}\ln \frac{2C_* }{\s}-\teta \ln \jap{i}
			\qquad\text{if}\quad \jap{i}< i_0\,,
		\end{array}
		\right.
	\end{equation*}
	where 
	$$
	i_0:=\left(\frac{2C_* }{\s}\right)^{\frac 2 \theta}\,,
	$$
	since the maximum is achieved for 
	$x=1$ if $\jap{i}\geq i_0$ and 
	$x=\frac{2C_* }{\s \jap{i}^{\theta/2}}$
	if $\jap{i}< i_0$.
	Note that $i_0\geq e.$
	Then we get 
	\begin{eqnarray*}
		&& \sum_i f_i(|\ell_i|)
		=
		\sum_{i\ \text{s.t.} \ \ell_i\neq 0} f_i(|\ell_i|)
		\leq
		\\
		&& \sum_{\jap{i}\geq i_0\ \text{s.t.} \ \ell_i\neq 0}
		\left(
		2\ln \jap{i} +1
		-\frac{\s}{C_* }\jap{i}^{\frac\theta 2}
		\right)
		+
		\sum_{\jap{i}< i_0}
		\left(
		2\ln \frac{2C_* }{\s}+\Big(2
		-\teta\Big) \ln \jap{i}
		\right)\,.
	\end{eqnarray*}
	We immediately have that
	\begin{eqnarray*}
		&& \sum_{\jap{i}< i_0}
		\left(
		2\ln \frac{2C_* }{\s}+\Big({2}
		-\teta\Big) \ln \jap{i}
		\right)
		\leq 6 i_0 
		\left(
		\ln \frac{2C_* }{\s}+
		 \ln i_0
		\right)
		\\
		&&
		=
		6\left(1+ \frac{2}{\theta}\right)\left(\frac{2C_* }{\s}\right)^{\frac 2 \theta}
		\ln \frac{2C_* }{\s}\,.
	\end{eqnarray*}
	Moreover, in the case $\jap{i}\geq i_0\geq e,$
	$$
	2\ln \jap{i} +1
	-\frac{\s}{C_* }\jap{i}^{\frac\theta 2}
	\leq 
3\ln \jap{i}
	-\frac{\s}{C_* }\jap{i}^{\frac\theta 2} 
	=\frac{6}{\theta}\Big(
	\ln \jap{i}^{\frac\theta 2}-{2\frak C } \jap{i}^{\frac\theta 2}
	\Big)
	$$
	where 
	$$
	{\frak C }:=\frac{\theta\s(2-2^\theta)}{84 }<1
	\,.
	$$
	We have that\footnote{Using that, for every fixed
		$0<{\frak C} \leq 1,$ we have
		${\frak C} x\geq \ln x$ for every $x\geq
		\frac{2}{{\frak C}}\ln
		\frac{1}{{\frak C}} .$}
	$$
	\ln \jap{i}^{\frac\theta 2}-2{\frak C} \jap{i}^{\frac\theta 2}
	\leq -{\frak C} \jap{i}^{\frac\theta 2}\,,
	\qquad
	\text{when}\qquad
	\jap{i}\geq 
	i_*:=
	\left(\frac{2}{{\frak C}}\ln
	\frac{1}{{\frak C}}\right)^{\frac2\theta}
	\,.
	$$
	Note that
	$$
	i_\sharp
	\geq \max\{ i_0, i_*\}\,.
	$$
	Therefore
	\begin{eqnarray*} & &\sum_{\jap{i}\geq i_0\ \text{s.t.} \ \ell_i\neq 0}\pa{2\ln \jap{i} +1
			-\frac{\s}{C_* }\jap{i}^{\frac\theta 2}}
		\leq
		\sum_{\jap{i}\geq i_0\ \text{s.t.} \ \ell_i\neq 0}
		\frac{6}{\theta}\pa{
			\ln \jap{i}^{\frac\theta 2}-{2\frak C } \jap{i}^{\frac\theta 2}} \\
		&\leq&
		\frac{6}{\theta}
		\left(
		\sum_{\jap{i}<i_\sharp}
		\ln \jap{i}^{\frac\theta 2}
		-
		\sum_{\jap{i}\geq i_\sharp\ \text{s.t.} \ \ell_i\neq 0}
		\Big(
		{\frak C} \jap{i}^{\frac\theta 2}
		\Big)
		\right)
	\le  9  i_\sharp \ln i_\sharp
		\,.
	\end{eqnarray*}
	In conclusion we get
	\begin{eqnarray*}
		\sum_i f_i(|\ell_i|)
		&\leq&
		6\frac{2+\theta}\theta\left(\frac{2C_* }{\s}\right)^{\frac 2 \theta}
		\ln \frac{2C_* }{\s}
		+
	9
		 i_\sharp \ln i_\sharp
		\\
		&\leq&		
		9 \left(\frac{2C_* }{\s\theta }\right)^{\frac 2 \theta}
		\ln \pa{\frac{2C_* }{\s}}^{\frac{\theta}{2}}
	+ 9 	i_\sharp \ln i_\sharp  \le 18 	i_\sharp \ln i_\sharp
	\end{eqnarray*}
\end{proof}

The inequality \eqref{cavolfiore gevrey} follows from plugging \eqref{gricia}
into \eqref{spa} and evaluating the constant.
\end{proof}
\newcommand{\etalchar}[1]{$^{#1}$}
\def\cprime{$'$}


\begin{thebibliography}{CKS{\etalchar{+}}10}

\bibitem[BBHM]{BBHM} Baldi P., Berti M., Haus E., Montalto R., \textit{Time quasi-periodic gravity water waves in finite depth}, 
Invent. math., {214}(2), {739-911}, 2018
\bibitem[Bam99a]{Bambu:1999}
D.~Bambusi.
\newblock {N}ekhoroshev theorem for small amplitude solutions in nonlinear
{S}chr\"odinger equations.
\newblock {\em Math. Z.}, 230(2):345--387, 1999.
%
\bibitem[Bam99b]{Bambu:1999b}
D.~Bambusi.
\newblock On long time stability in {H}amiltonian perturbations of nonresonant
linear {PDE}'s.
\newblock {\em Nonlinearity}, 12:823--850, 1999.

\bibitem[Bam03]{Bambusi:2003}
D.~Bambusi.
\newblock Birkhoff normal form for some nonlinear {PDE}s.
\newblock {\em Comm. Math. Phys.}, 234(2):253--285, 2003.
%
\bibitem[BDGS07]{BDGS}
D.~Bambusi, J.-M. Delort, B.~Gr\'{e}bert, and J.~Szeftel.
\newblock Almost global existence for {H}amiltonian semilinear {K}lein-{G}ordon
equations with small {C}auchy data on {Z}oll manifolds.
\newblock {\em Comm. Pure Appl. Math.}, 60(11):1665--1690, 2007.

\bibitem[BG03]{Bambusi-Grebert:2003}
D.~Bambusi and B.~Gr\'{e}bert.
\newblock Forme normale pour {NLS} en dimension quelconque.
\newblock {\em C. R. Math. Acad. Sci. Paris}, 337(6):409--414, 2003.
%
\bibitem[BG06]{Bambusi-Grebert:2006}
D.~Bambusi and B.~Gr\'ebert.
\newblock {B}irkhoff normal form for partial differential equations with tame
modulus.
\newblock {\em Duke Math. J.}, 135(3):507--567, 2006.
\bibitem[BS20]{stolo-bambusi}
D. Bambusi and L. Stolovitch.
\newblock Convergence to normal forms of integrable {PDE}s.
\newblock {\em Comm. Math. Phys.}, 376(2):1441--1470, 2020.

%

\bibitem[BCG]{BCG}
G.Benettin, L. Chierchia, and
M. Guzzo, 
\newblock The Steep Nekhoroshev’s Theorem.  
\newblock {\em Commun. Math. Phys.}, 342, 569–601 (2016)

\bibitem[BGG85]{BGG85}
G.~Benettin, L.~Galgani and A.~Giorgilli.
\newblock A proof of Nekhoroshev's theorem for the stability times in nearly integrable Hamiltonian systems.
\newblock {\em Celestial Mech.}, 37(1): 1--25, 1985.

\bibitem[BFG88]{BenFroGior}
G.~Benettin, J.~Fr{\"o}hlich, and A.~Giorgilli.
\newblock A {N}ekhoroshev-type theorem for {H}amiltonian systems with
infinitely many degrees of freedom.
\newblock {\em Comm. Math. Phys.}, 119(1):95--108, 1988.



\bibitem[BKM18]{BKM18}
M. Berti, Th. Kappeler, and R. Montalto.
\newblock Large {KAM} tori for perturbations of the defocusing {NLS} equation.
\newblock {\em Ast\'{e}risque}, (403):viii+148, 2018.

	
\bibitem[BBP13]{BBiP1}
M.~Berti, L.~Biasco, and M.~Procesi.
\newblock {KAM} theory for the {H}amiltonian derivative wave equation.
\newblock {\em Annales Scientifiques de l'ENS}, 46(2):299--371, 2013.
%

\bibitem[BB15]{BB3}
M. Berti and Ph. Bolle.
\newblock Quasi-Periodic Solutions of Nonlinear Wave Equations on the d-Dimensional Torus
\newblock {\em EMS Series Lect. in Math}, 2020

\bibitem[BBP10]{berti-bolle-abstract} M.~Berti, P.~Bolle, and M.~Procesi. %
\newblock An abstract {N}ash-{M}oser theorem with parameters and
applications to {PDE}s. 
\newblock {\em Ann. Inst. H. Poincar\'e Anal. Non
	Lin\'eaire}, 27(1):377--399, 2010.

\bibitem[BCP]{BCP} M.Berti, L.Corsi, M.Procesi. 
\newblock An abstract Nash-Moser
theorem and quasi-periodic solutions for NLW and NLS on compact Lie groups
and homogeneous manifolds.
\newblock  {\em Com. Math. Phys.} 2014. DOI 10.1007/s00220-014-2128-4

\bibitem[BFP] {BFP} Berti M., Feola R., Pusateri F.,
\newblock  Birkhoff normal form and long time existence for periodic gravity Water Waves, preprint arXiv:1810.11549.

\bibitem[BD18]{Berti-Delort}
M.~Berti and J.M. Delort.
\newblock {\em {A}lmost global existence of solutions for capillarity-gravity
	water waves equations with periodic spatial boundary conditions}.
\newblock Springer, 2018.
%
\bibitem[BM21]{BM} M. Berti, L. Franzoi, A. Maspero.
\newblock{ Traveling quasi-periodic water waves with constant vorticity.}
\newblock{ \em ARMA }240(1): 99-202, 2021.
%

\bibitem[BMP18]{BMP18}
L.~Biasco, J.E. Massetti, and M.~Procesi.
\newblock {A}n abstract {B}irkhoff {N}ormal {F}orm {T}heorem and exponential
type stability of the 1d {NLS}.
\newblock {\em Comm. Math. Phys.}, 375(3),  2089--2153, 2020.

\bibitem[BMP21]{BMP:almost}
L.~Biasco, J.E. Massetti, and M.~Procesi.
\newblock Almost-periodic invariant tori for the {NLS} on the circle.
\newblock {\em Ann. Inst. H. Poincar\'{e} Anal. Non Lin\'{e}aire},
38(3):711--758, 2021.

%

%
%
\bibitem[BFN15]{Bounemoura-Fayad-Niederman:2015}
A.~Bounemoura, B.~Fayad, and L.~Niederman.
\newblock Double exponential stability for generic real-analytic elliptic
equilibrium points.
\newblock 2015.
\newblock Preprint ArXiv : arxiv.org/abs/1509.00285.
%


\bibitem[Bou98]{Bo98}
J. Bourgain.
\newblock  Quasi-periodic solutions of Hamiltonian perturbations of 2D linear
Schr\"odinger equations.
\newblock{\em Ann. of Math.} (2), 148(2):363–439, 1998.
\bibitem[Bou96a]{Bourgain:1996}
J.~Bourgain.
\newblock Construction of approximative and almost periodic solutions of
perturbed linear {S}chr\"odinger and wave equations.
\newblock {\em Geom. Funct. Anal.}, 6(2):201--230, 1996.


\bibitem[Bou05]{Bourgain:2005}
J.~Bourgain.
\newblock On invariant tori of full dimension for 1{D} periodic {NLS}.
\newblock {\em J. Funct. Anal.}, 229(1):62--94, 2005.
%
\bibitem[Bru72]{bruno}
A.D. Bruno.
\newblock {Analytical form of differential equations}.
\newblock {\em Trans. Mosc. Math. Soc}, 25,131-288(1971); 26,199-239(1972),
1971-1972.
%
\bibitem[CM18]{CM}
L. Corsi, R. Montalto
\newblock{Quasi-periodic solutions for the forced Kirchhoff equation on $\T^d$}
\newblock{\em NonLinearity}, 31(11), 5075-5109.
\bibitem[CLSY]{Yuan_et_al:2017}
H.~Cong, J.~Liu, Y.~Shi, and X.~Yuan.
\newblock The stability of full dimensional KAM tori for nonlinear Schrödinger
equation.
\newblock {\em J. Differential Equations}, 264(7):4504--4563, 2018.

\bibitem[CMW]{Cong}
H.~Cong, L.~Mi, and P.~Wang.
\newblock A {N}ekhoroshev type theorem for the derivative nonlinear
{S}chr\"odinger equation.
\newblock {\em J. Differential Equations}, 268(9):5207--5256, 2020.

\bibitem[CY21]{CY}Cong, H., Yuan, X.
\newblock The existence of full dimensional invariant tori for 1-dimensional nonlinear wave equation
\newblock {\em Annales de l'Institut Henri Poincare (C) Analyse Non Lineaire}, 2021, 38(3), pp. 759–786

\bibitem[CW93]{CW}
W.~Craig and C.~E. Wayne.
\newblock Newton's method and periodic solutions of nonlinear wave equations.
\newblock {\em Comm. Pure Appl. Math.}, 46(11):1409--1498, 1993.




%
\bibitem[DS04]{DS0}
J.-M. Delort and J.~Szeftel.
\newblock Long-time existence for small data nonlinear {K}lein-{G}ordon
equations on tori and spheres.
\newblock {\em Int. Math. Res. Not.}, (37):1897--1966, 2004.

\bibitem[DS06]{DS}
J.-M. Delort and J.~Szeftel.
\newblock Bounded almost global solutions for non {H}amiltonian semi-linear
{K}lein-{G}ordon equations with radial data on compact revolution
hypersurfaces.
\newblock {\em Ann. Inst. Fourier (Grenoble)}, 56(5):1419--1456, 2006.
%

\bibitem[Del12]{Delort-2009}
J.~M. Delort.
\newblock A quasi-linear birkhoff normal forms method. application to the
quasi-linear klein-gordon equation on $\mathtt{S}^1$.
\newblock {\em Ast\'erisque}, 341, 2012.
\bibitem[D15]{DS15}
J.-M. Delort
\newblock Quasi-Linear Perturbations of Hamiltonian Klein-Gordon Equations on Spheres,
\newblock{ American Mathematical Society}, 2015.

\bibitem[EK10]{EK} 
L. H. Eliasson and S. B. Kuksin. 
\newblock KAM for the nonlinear Schr\"odinger equation.
\newblock{\em  Ann. of
Math.} (2), 172(1):371–435, 2010.
\bibitem[FG13]{Faou-Grebert:2013}
E.~Faou and B.~Gr\'ebert.
\newblock A {N}ekhoroshev-type theorem for the nonlinear {S}chr\"odinger
equation on the torus.
\newblock {\em Anal. PDE}, 6(6):1243--1262, 2013.
%
%
\bibitem[FG]{FG}
 R. Feola - F. Giuliani,
\newblock Quasi-periodic Traveling Waves on an Infinitely Deep Perfect Fluid Under Gravity,
  preprint, arXiv:2005.08280 (2020).

\bibitem[FGP]{FGP} Feola R., Giuliani F., Pasquali S., \textit{On the integrability of Degasperis-Procesi equation: control of the 
	Sobolev norms and Birkhoff resonances}, 	
\newblock {\em J. Differential Equations } 266 (2019) 3390–3437

\bibitem[FGPr]{FGPr}R. Feola - F. Giuliani - M. Procesi, 
\newblock Reducibile KAM tori for the Degasperis Procesi equation,
\newblock{\em  Comm. Math. Phys.}, 377, (2020), 1681–1759.
\bibitem[FI18]{FI}
R.~Feola and F.~Iandoli.
\newblock Long time existence for fully nonlinear NLS with small Cauchy data on
the circle.
\newblock Ann. Sc. Norm. Super. Pisa Cl. Sci. (5) Vol. XXII (2021), 109-182

\bibitem[FI20]{FI2}
R.~Feola and F.~Iandoli.
\newblock A non-linear Egorov theorem and Poincar\'e-Birkhoff normal forms for quasi-linear pdes on the circle. 
\newblock preprint 2020, arxiv:2002.1244
\bibitem[GYX]{GYX}
Geng, J., Xu, X., You, J.
\newblock An infinite dimensional KAM theorem and its application to the two dimensional cubic Schr\"odinger equation
\newblock{ \em Advances in Mathematics}, 2011, 226(6), pp. 5361–5402

\bibitem[KP03]{kappeler-poschel-book}
Th. Kappeler and J. P\"{o}schel.
\newblock {Kd{V} \& {KAM}},
\newblock  {\em Ergebnisse der Mathematik und
	ihrer Grenzgebiete. 3.}, 45,
Springer-Verlag, Berlin, 2003.

\bibitem[K88]{Kuk}
S.~B. Kuksin.
\newblock Perturbation of conditionally periodic solutions of
infinite-dimensional {H}amiltonian systems.
\newblock {\em Izv. Akad. Nauk SSSR Ser. Mat.}, 52(1):41--63, 240, 1988.

\bibitem[KP96]{KP}
Kuksin, S. and Pöschel, J.
\newblock Invariant cantor manifolds of quasi-periodic oscillations for a nonlinear Schrödinger equation
\newblock{Annals of Mathematics}, 143(1) 149-179 {1996}.

	\bibitem[P\"os90]{Po}
			{Pöschel, J.},
		\newblock{Small divisors with spatial structure in infinite dimensional Hamiltonian systems}
		\newblock{Communications in Mathematical Physics}, 127(2), 351-393, {1990}.
		
		


\bibitem[IoL05]{lombardi-nf}
G.~Iooss and E. Lombardi.
\newblock Polynomial normal forms with exponentially small remainder for analytic vector fields.
\newblock {\em J. Differential Equations}, 212(1):1--61, 2005.

\bibitem[K00]{kuksin-book}
 S.~B.~Kuksin.
\newblock Analysis of {H}amiltonian {PDE}s.
\newblock {\em Oxford Lecture Series in Math. and Appl.}, 19,
Oxford University Press, Oxford, 2000.

\bibitem[LS10]{stolo-lombardi}
E.~Lombardi and L.~Stolovitch.
\newblock Normal forms of analytic perturbations of quasihomogeneous vector
fields: Rigidity, invariant analytic sets and exponentially small
approximation.
\newblock {\em Ann. Scient. Ec. Norm. Sup.}, pages 659--718, 2010.

\bibitem[N77]{Nekho}
N.N.~Nehorošev.
\newblock An exponential estimate of the time of stability of nearly integrable Hamiltonian systems.
\newblock {\em Uspehi Mat. Nauk}, 32, 6(198),  5--66, 1977.

\bibitem[Ni04]{Nid}
Niederman L.
\newblock Exponential stability for small perturbations of steep integrable Hamiltonian systems.
\newblock{\em Ergod. Theory Dyn. Syst.}, 24(2), 593–608 (2004)	.

\bibitem[Nik86]{Nik}
N.V. Nikolenko.
\newblock The method of {P}oincar\'e normal forms in problems of integrability
of equations of evolution type.
\newblock {\em Russian Math. Surveys}, 41:5:63--114, 1986.

\bibitem[P{\"o}s86]{Posch}
J.~P{\"o}schel.
\newblock {On invariant manifolds of complex analytic mappings near fixed
	points}.
\newblock {\em Expo. Math.}, 4,97-109, 1986.

\bibitem[P{\"o}s99]{Posch-Nekho}
J.~P{\"o}schel.
\newblock {On Nekhoroshev's estimate at an elliptic equilibrium}.
\newblock {\em Internat. Math. Res. Notices}, (4), 203--215, 1999.
Pöschel, Jürgen

\bibitem[PP16]{PPBumi} C.  Procesi and M. Procesi, 
\newblock Reducible quasi-periodic solutions for the non linear Schr\"odinger equation,
\newblock {\em BUMI} 9(2),189, 2016

\bibitem[Rue77]{russmann-ihes}
H. R{\"u}ssmann.
\newblock On the convergence of power series transformations of analytic mappings near a fixed point into a normal form.
\newblock {\em Preprint I.H.E.S.},  M/77/178, 1--44, 1977.

\bibitem[Se92]{serre-Lie1}
J.-P~ Serre.
\newblock {Lie Algebras and Lie groups}.
\newblock {\em Lecture Notes in Mathematics}, 1500, Springer-Verlag, 1992

\bibitem[Sie42]{siegel}
C.L. Siegel.
\newblock Iteration of analytic functions.	
\newblock  {\em Ann. of Math.} (2), 43:607--612, 1942.

\bibitem[Sto00]{Stolo-ihes}
L.~Stolovitch.
\newblock Singular complete integrabilty.
\newblock {\em Publ. Math. I.H.E.S.}, 91:133--210, 2000.

\bibitem[Sto15]{stolo-diffeos}
L.~Stolovitch.
\newblock Family of intersecting totally real manifolds of {$(\Bbb C^n,0)$} and
germs of holomorphic diffeomorphisms.
\newblock {\em Bull. Soc. math. France}, 143(1):247--263, 2015.


%
\bibitem[Way90]{W}
C.~E. Wayne.
\newblock Periodic and quasi-periodic solutions of nonlinear wave equations via
{KAM} theory.
\newblock {\em Comm. Math. Phys.}, 127(3):479--528, 1990.
\bibitem[Y21] {Y} Yuan, X.,
\newblock KAM Theorem with Normal Frequencies of Finite Limit-Points for Some Shallow Water Equations
\newblock {CPAM}, 2021, 74(6), pp. 1193–1281
\bibitem[YZ14]{Yuan-Zhang}
Xiaoping Yuan and Jing Zhang.
\newblock Long time stability of {H}amiltonian partial differential equations.
\newblock {\em SIAM J. Math. Anal.}, 46(5):3176--3222, 2014.
%
\bibitem[Zeh77]{zehnder-infinite}
E. Zehnder.
\newblock C. {L}. {S}iegel's linearization theorem in infinite dimensions.
\newblock {\emph Manuscripta Math.}, 23(4) 363--371, 1977/78

\end{thebibliography}
\end{document}